\documentclass[11pt,letterpaper]{amsart}

\usepackage{graphicx}
\usepackage{booktabs}
\usepackage{listings}

\usepackage[latin1]{inputenc}
\usepackage{amsmath, amsthm,  amsfonts, amscd}
\usepackage{amssymb} 
\usepackage{url}
\usepackage{mathdots, arydshln, leftidx, yhmath} 
\newtheorem{lemma}{Lemma}
\newtheorem{theorem}[lemma]{Theorem}
\newtheorem{prop}[lemma]{Proposition}

\newtheorem{corollary}[lemma]{Corollary}
\numberwithin{equation}{section} \numberwithin{lemma}{section}

\usepackage{cite}

\newcommand{\kommentar}[1]{}

\newcommand{\Q}{\mathbb{Q}}
\newcommand{\Z}{\mathbb{Z}}
\newcommand{\C}{\mathbb{C}}

\newcommand{\sgn}{{\text{sgn}}}
\newcommand{\be}{\begin{equation}}
\newcommand{\ee}{\end{equation}}
\newcommand*{\reff}[1]{(\ref{#1})}
\newcommand{\kr}[2]{\pr{\fr{#1}{#2}}}
\newcommand{\pr}[1]{\!\left(#1\right)}
\newcommand{\fr}[2]{\frac{#1}{#2}}

\newcommand{\lr}[1]{\left|\;\!#1\;\!\right|}

\newtheorem{proposition}[lemma]{Proposition}

\begin{document}

\title
  [LOWER ORDER TERMS]
  {Lower order terms for the moments of symplectic and orthogonal families of $L$-functions}
\author[Goulden]{Ian P. Goulden}\email{ipgoulden@uwaterloo.ca}

\author[Huynh]{Duc Khiem Huynh}\email{dkhuynhms@gmail.com}

\author[Rishikesh]{Rishikesh}\email{rishikes@gmail.com}

\author[Rubinstein]{Michael O. Rubinstein}
\thanks{Support for work on this paper was provided by the
National Science Foundation under award DMS-0757627 (FRG grant),
and two NSERC Discovery Grants.}
\email{mrubinst@uwaterloo.ca}

\address{Department of Pure Mathematics, University of Waterloo, Waterloo, ON, N2L 3G1, Canada}

\date{\today}

\thispagestyle{empty} \vspace{.5cm}
\begin{abstract}
We derive formulas for the terms in the conjectured asymptotic expansions of
the moments, at the central point, of quadratic Dirichlet $L$-functions,
$L(1/2,\chi_d)$, and also of the $L$-functions associated to quadratic twists
of an elliptic curve over $\Q$. In so doing, we are led to study determinants
of binomial coefficients of the form $\det \left( \binom{2k-i-\lambda_{k-i+1}}{2k-2j}\right)$.
\end{abstract}

\maketitle
\tableofcontents

\section{Introduction}

In this paper we describe formulas, derived from conjectures of Conrey, Farmer,
Keating, Rubinstein, and Snaith \cite{CFKRS}, for the moments of quadratic
Dirichlet $L$-functions at the central point, and the moments of $L$-functions
associated to quadratic twists of an elliptic curve.

We are motivated to study moments in these two families of $L$-functions
because of their apparent connection to the moments of characteristic
polynomials of unitary symplectic and orthogonal matrices.

Montgomery was the first to discover a link between an $L$-function and
characteristic polynomials of unitary matrices~\cite{M}. He
computed, with restrictions on the allowed test functions, the limiting pair
correlation of the zeros of the Riemann zeta function, and found that it coincides with
the average pair correlation of the eigenvalues of large random (according to Haar measure)
unitary matrices that had been computed earlier by Dyson~\cite{Dy}. Odlyzko
later confirmed this agreement numerically, without restriction~\cite{O}.
Rudnick and Sarnak generalized Montgomery's result to higher correlations and
to any primitive $L$-function~\cite{RS}.

Katz and Sarnak then made precise connections between
various families of $L$-functions and matrices from specific classical compact
groups, based on results linking the density of zeros $L$-functions and
analogous zeta functions over function fields, to the eigenvalue densities of
random matrices in the classical compact groups~\cite{KS}~\cite{KS2}. For
instance, their work showed a statistical connection between the zeros of
quadratic Dirichlet $L$-functions and eigenvalues of unitary symplectic
matrices, and between the zeros of $L$-functions of quadratic twists of an
elliptic curve and eigenvalues of orthogonal matrices. The papers~\cite{R}
and~\cite{R2}, provided further theoretical and numerical support for the
relevance of these matrix groups to our two families of $L$-functions.

Subsequently, Keating and Snaith were able to predict the leading term in the
asymptotics for the moments of the Riemann zeta function on the critical line by carrying
out an analogous computation for random unitary matrices~\cite{KeS}. In a
companion paper~\cite{KeS2}, they also conjectured the leading term in the
asymptotics for the moments in our two families of $L$-functions by computing the moments of
the characteristic polynomials of random unitary symplectic and even orthogonal
matrices. See also the paper of Conrey and Farmer~\cite{CF} which contains some
arithmetic information needed for the Keating and Snaith approach to moments.

The method of Keating and Snaith for predicting moments of $L$-functions relies
on computations in random matrix theory, for example it uses Weyl's integration
formula and the Selberg integral, and some guesswork to make the heuristic leap
to number theoretic moments. It also has the drawback of requiring, as input,
the relevant classical compact group as predicted by Katz and Sarnak.

On the other hand, the approach, referred to above, of Conrey, Farmer, Keating,
Rubinstein, and Snaith does not rely on random matrix theory to derive,
heuristically, the moments of various families of $L$-functions. Their method
is based strictly on number theoretic techniques involving the approximate
functional equation, the traditional equation that is used to study moments of
$L$-functions~\cite{T}~\cite{J}. While random matrix theory is not
needed in their approach, the formulas that their heuristic approach yields for
$L$-functions have provable analogues in random matrix theory. CFKRS were also
able to make progress by introducing `shifts' into the moments, a strategy that
was inspired by Motohashi's evaluation of the fourth moment of the zeta function
~\cite{Mot} and also by an analogous  problem in random matrix theory.
Their method, therefore, produces an answer that can be compared against
various moment computations in random matrix theory, and, instead of using the
predictions of Katz and Sarnak, it provides evidence for them. Furthermore, the
conjectured formulas of CFKRS go well beyond the leading asymptotic of Keating
and Snaith, providing, implicitly, a full asymptotic expansion for a variety of
$L$-function moment problems. Because their conjectured formulas provide a full asymptotic
expansion for moments, one can test them numerically by comparing the predicted moments
against those computed from $L$-function data. See for instance~\cite{CFKRS}
\cite{CFKRS2} \cite{AR} \cite{RY}.

Our goal is to turn the implicit formulas of CFKRS
into asymptotic expansions with explicitly given coefficients. We elaborate
on the CFKRS formulas, for the family of quadratic Dirichlet $L$-functions, in
Section~\ref{sec:CFKRS quadratic} and for quadratic elliptic curve
$L$-functions in Section~\ref{sec:elliptic}.

Besides the approaches of Keating and Snaith and of CFKRS, two additional
methods have yielded interesting results for the moments of $L$-functions.

Gonek, Hughes, and Keating~\cite{GHK}, and Bui and Keating~\cite{BK} use the
explicit formula for an $L$-function to realize the $L$-function as a hybrid
between partial Hadamard and Euler products. They assume statistical
independence between these two products and study the moments of the partial
Euler product using number theoretic heuristics. The moments of the partial
Hadamard product are studied by modeling the zeros of the Hadamard product
based on the predicted classical compact group. Their approach therefore
suffers the same disadvantage of the Keating and Snaith method of requiring the
predictions of Katz and Sarnak as input. The main advantage of their method
over the Keating and Snaith method is that it explains, rather than guesses,
the appearance of an `arithmetic factor' in moment formulas for $L$-functions.
Another disadvantage is that it only seems to correctly predict the leading
asymptotic for the $L$-function moments that they consider, and thus only
agrees with the CFKRS prediction to leading order. Presumably this is because
their assumptions are too strong, for example the statistical independence
between the partial Hadamard and Euler products, and their use of matrix
eigenvalues to model the partial Hadamard product.

Another method for studying moments of $L$-functions has been developed by
Diaconu, Goldfeld, and Hoffstein~\cite{GH}~\cite{DGH} and uses the theory of
multiple Dirichlet series. It has the advantage of proving asymptotic formulas
for some $L$-function moments, for example the first three moments of quadratic
Dirichlet $L$-functions at the central point. However, it has the disadvantage
of involving an elaborate sieving process (in the case of quadratic
characters), that makes it unwieldy for producing explicit formulas for the
asymptotic expansion. Interestingly, their method predicts the existence of
additional lower order terms of smaller magnitude that go beyond those of the
asymptotic expansion of CFKRS. See the paper of DGH as well as that of
Zhang~\cite{Z}, and Alderson and Rubinstein~\cite{AR} for discussions and
computations regarding these additional lower terms.

\subsection{The CFKRS conjecture for $L(1/2,\chi_d)$}
\label{sec:CFKRS quadratic}

We begin by describing the CFKRS conjecture for quadratic Dirichlet $L$-functions.
Let $D$ be a squarefree integer, $D\neq 0,1$, and let $K= \Q(\sqrt{D})$
be the corresponding quadratic field. The fundamental discriminant $d$ of $K$ equals
$D$ if $D = 1 \mod 4$, and $4D$ if $D = 2,3 \mod 4$.
Let $\chi_d(n)$ be the Kronecker symbol $\kr{d}{n}$,
and $L\pr{s,\chi_d}$ the quadratic Dirichlet $L$-function given by the
Dirichlet series
\begin{equation}
    L\pr{s,\chi_d} = \sum_{n=1}^\infty \fr{\chi_d(n)}{n^s}, \qquad \Re(s)>0,
\end{equation}
satisfying the functional equation
\begin{equation}
    L\pr{s,\chi_d} = \lr{d}^{\fr{1}{2} - s} X(s,a) L\pr{1-s, \chi_d},
\end{equation}
where
\begin{equation}
    \label{eq:X}
    X(s,a) =
    \pi^{s-\fr{1}{2}} \fr{\Gamma\pr{\fr{1-s+ a}{2}}}{\Gamma\pr{\fr{s+a}{2}}}, \qquad a =
    \begin{cases} 0 \quad &\text{if $d> 0$,} \\ 1 \quad &\text{if $d<0$.} \end{cases}
\end{equation}

Let $S(X)$ denote the set of fundamental discriminants with $|d|<X$. The Gamma factor
in functional equation for $L(s,\chi_d)$ depends on whether $d<0$ or $d>0$. Thus, define
further
\begin{eqnarray}
    S_+(X) &=& \{ d \in S(X) : d>0\} \notag \\
    S_-(X) &=& \{ d \in S(X) : d<0\},
\end{eqnarray}
to be, respectively, the sets of positive and negative fundamental discriminants with $|d|<X$.

CFKRS conjectured \cite{CFKRS} the asymptotic expansion:
\begin{equation}
    \label{eq: moment asympt}
    \sum_{d \in S\pm(X)} L\pr{1/2,\chi_d}^k \sim
    \frac{3}{\pi^2} X  \mathcal Q_\pm(k,\log X),
\end{equation}
where $\mathcal Q_+(k,x)$ and $\mathcal Q_-(k,x)$ are polynomials of degree
$k(k+1)/2$ in $x$ that we will describe below. The fraction $3/\pi^2$ accounts
for the density of fundamental discriminants amongst all the integers.

The polynomial $\mathcal Q_\pm(k,\log X)$ is expressed
in terms of a more fundamental polynomial $Q_\pm(k,x)$ of the same degree
that captures the moments locally:
\begin{equation}
      \mathcal Q_\pm(k,\log{X}) = \frac{1}{X} \int_1^X Q_\pm(k,\log{t}) dt.
  \end{equation}
One of the main achievements of CFKRS was to give a general recipe/heuristic
for producing formulas for moments of various families of $L$-functions. Their
formula (see Conjecture 1.5.3 in \cite{CFKRS}) for the polynomial
$Q_\pm(k,x)$ is given implicitly in terms of a $k$-fold multivariate residue:
\begin{equation}
    \label{eq:Qkresidue}
    Q_\pm(k,x)= \frac{(-1)^{k(k-1)/2} 2^k}{k!} \frac {1}{(2\pi i)^k}
    \oint \cdots \oint
    \frac{G_\pm(z_1,\ldots,z_k)\Delta(z_1^2,\ldots,z_k^2)^2}{\prod_{j=1}^k z_j^{2k-1}}
    e^{\frac x 2 \sum_{j=1}^k z_j} \,  dz_1\ldots dz_k,
  \end{equation}
where $\Delta(w_1,\dotsc,w_k)$ is the Vandermonde determinant
\begin{equation}
    \label{eq:1}
    \Delta(w_1,\dotsc,w_k) = \det(w_i^{j-1})_{k\times k}
    =\prod_{1\leq i < j \leq k} (w_j -w_i),
\end{equation}
and 
\begin{equation}
    \label{eq:Gdefinition}
    G_\pm(z_1,\ldots,z_k) = A_k(z_1,\ldots,z_k) \prod_{j=1}^k X(\frac 1 2 +z_j,
    a)^{-1/2} \prod_{1\leq i\leq j \leq k} \zeta(1+z_i+z_j).
\end{equation}
Here, $a=0$ for $G_+$ and $a=1$ for $G_-$, $X(s,a)$ is given in~\eqref{eq:X},
and $A_k$ equals the Euler product, absolutely convergent in a neighbourhood
of $(z_1,\ldots,z_k)=(0,\ldots,0)$, defined by
\begin{multline}
    \label{eq:Ak}
    \index{$A_k$}A_k(z_1, \ldots, z_k) = \prod_p \prod_{1\leq i\leq j \leq k}
    \left(1-\frac{1}{p^{1+z_i+z_j}} \right)\\
    \times \left(\frac 1 2 \left( \prod_{j=1}^{k}\left(1-\frac 1
          {p^{\frac 1 2 +z_j}} \right)^{-1} +
        \prod_{j=1}^{k}\left(1+\frac 1{p^{\frac 1 2 +z_j}}
        \right)^{-1} \right) +\frac 1 p \right) \left(1+\frac 1 p
    \right)^{-1}.
\end{multline}

One advantage of equation~\eqref{eq:Qkresidue} is that it allows one to easily
see that $Q_\pm(k,x)$ is a polynomial of degree $k(k+1)/2$ in $x$. That is
because the denominator of the multivariate residue picks up terms in the
numerator involving $\prod_{j=1}^k z_j^{2k-2}$, which is of degree $2k(k-1)$.
Now, the factor $\Delta(z_1^2,\ldots,z_k^2)^2$ is a homogeneous polynomial,
also of degree $2k(k-1)$. However, the factor $G_\pm(z_1,\ldots,z_k)$ cancels
$k(k+1)/2$ of the factors of the Vandermonde, because each $\zeta(1+z_i+z_j)$
has a Laurent expansion that begins $1/(z_i+z_j)$ coming from the pole at $s=1$
of $\zeta(s)$. Therefore, in considering the multivariate Taylor expansion of
the numerator about $z_1=\ldots=z_k=0$, we only need to take terms in the
series
\begin{equation}
    \label{eq:exp}
    \exp\left(\frac x 2 \sum_{j=1}^k z_j\right) = \sum_{0}^\infty \frac{x^n}{2^nn!} (z_1+\ldots +z_k)^n
\end{equation}
up to $n=k(k+1)/2$. Hence, in the $x$ aspect, the $k$-fold
residue only involves terms up to $x^{k(k+1)/2}$.

Equation~\eqref{eq:Qkresidue} has the disadvantage of expressing $Q_\pm(k,x)$ implicitly.
Let us therefore write
\begin{equation}
    \label{eq:Qexpansion}
    Q_\pm(k,x) = c_\pm(0,k) x^{k(k+1)/2} + c_\pm(1,k) x^{k(k+1)/2 - 1} + \ldots + c_\pm(k(k+1)/2,k).
\end{equation}
Our main result, described in the following theorem, gives explicit formulae for the
coefficients $c_{\pm}(r,k)$. We first define
\begin{equation}
    \label{eq:a_k}
    a_k := A_k(0,\ldots,0) =
    \prod_p 
    \frac{\left(1 - \frac{1}{p}\right)^{\frac{k(k+1)}{2}}}
       {1 + \frac{1}{p}}
    \left(
       \frac{
           \left( 1 - \frac{1}{\sqrt{p}}\right)^{-k}
           +
           \left(1 + \frac{1}{\sqrt{p}}\right)^{-k}
       }{2} + 
       \frac{1}{p}
    \right).
\end{equation}

\begin{theorem}\label{thm:maintheorem}
In \reff{eq:Qexpansion}, the leading coefficient $c_\pm(0,k)$ of $Q_+(k,x)$ or
$Q_-(k,x)$ are both equal to
\begin{equation}
  \label{eq:146}
  \frac{  a_k}{ 2^{k} }
  \prod_{j=0}^{k-1}\frac{(2j)!}{(k+j)!}=:
  c(0,k),
\end{equation}
and, for given $r\geq 1$, we have  
\begin{equation}
  \label{eq:147}
    \index{$c_\pm(r,k)$}c_\pm(r,k) = c(0,k)  \sum_{|\lambda|=r} b^\pm_\lambda(k ) N_\lambda(k),
\end{equation}
where $N_\lambda(k)$ is a polynomial in $k$ of degree at most $2|\lambda|$, defined in~\eqref{eq:N},
$a_k$ is defined in~\eqref{eq:a_k}, and the
$b^{\pm}_\lambda(k)$'s are  the Taylor coefficients of a holomorphic function, defined in
\eqref{eq:powerseries} and \eqref{eq:145}.
The sum is over all partitions $|\lambda| = r$, with $\sum \lambda_i = r$ and
$\lambda_1\geq \lambda_2\geq \ldots > 0$.
\end{theorem}

We remark that formula~\eqref{eq:146} for the leading term agrees with the
prediction of Keating and Snaith. See (34),(45), and (47) of Keating and Snaith
\cite{KS2} (replacing $\log{D}$ by $x$ in their equation (45)). Their
derivation is heuristic and based on the Selberg integral. Compare also to the
leading term of equation (1.5.17) of~\cite{CFKRS}, with $N=x/2$ in that
equation. To verify the agreement between these, one can check, inductively,
that:
\begin{equation}
    \label{eq:comparison to KS}
    \frac{1}{2^k}\prod_{j=0}^{k-1}\frac{(2j)!}{(k+j)!}
    = \frac{1}{2^{k(k+1)/2}} \prod_{j=1}^k \frac{1}{(2j-1)!!}
    = \prod_{j=1}^k \frac{j!}{(2j)!}.
\end{equation}
Note that~\eqref{eq:147} is analogous to formula (1.16) of~\cite{CFKRS} which
provides a formula for the coefficients of the moment polynomials of the
Riemann zeta function. See also Dehaye's paper~\cite{D}, also for the Riemann zeta
function, where he gives a combinatorial formula for the analogue of our
polynomial $N_\lambda(k)$.

We work out examples, for $r=1$ and $r=2$.
Table~\ref{tab:N_lambda} provides $N_{(1)}(k) = k(k+1)$,
$N_{(1,1)}(k)=\frac 1 2 k (k-1)(k+1)(k+2)$, and $N_{(2)}(k)=0$. Thus,
\begin{align}
    \label{eq:c1}
    c_\pm(1,k) &= c(0,k) b^\pm_{(1)}(k) N_{(1)}(k) \notag \\
    &= \frac{a_k}{2^k}\prod_{j=0}^{k-1}\frac{(2j)!}{(k+j)!}\  k(k+1) b^\pm_{(1)}(k).
\end{align}
and
\begin{align}
    c_\pm(2,k) &= c(0,k)\left( b^\pm_{(1,1)}(k)N_{(1,1)}(k) +
    b^\pm_{(2)}(k)N_{(2)}(k)\right) \notag \\
    &=\frac{a_k}{2^k}\prod_{j=0}^{k-1}\frac{(2j)!}{(k+j)!}
    \times  \frac 1 2 k(k-1)(k+1)(k+2) \  b^\pm_{(1,1)}(k).
\end{align}
Let
\begin{equation}
    \label{eq:zeta gammas}
    \zeta(1+s) = \frac{1}{s} + \sum_{n=0}^\infty (-1)^n \gamma_n \frac{s^n}{n!},
\end{equation}
be the Laurent expansion about 0 of $\zeta(1+s)$ ($\gamma_0$ is Euler's constant),
and define
\begin{equation}
    \label{eq:f}
    f_j(p) :=
    \frac{(-1)^{j}(p^{1/2}-1)^{-j-k} + (p^{1/2}+1)^{-j-k}}
    {(p^{1/2}-1)^{-k}+(p^{1/2}+1)^{-k}+ 2 p^{-1-k/2}},
\end{equation}

Formulas for the coefficients $b_{(1)}^\pm(k)$, $b_{(1,1)}^\pm(k)$ can be derived
using the method described in Section~\ref{sec:b}, and are given by
\begin{equation}
    \label{eq:b1}
    b_{(1)}^\pm(k) = -\frac 1 2 \log \pi + \frac{1}{2} \frac{\Gamma'}{\Gamma} \left(1/4+a/2 \right)
    + (k+1)\gamma_0
    +\sum_p
    \left(
        \frac{(k+1)}{p-1} + f_1(p)
        \right) \log{p},
    \end{equation}
where $a=0$ for $b^+_{(1)}$, and $a=1$ for $b^-_{(1)}$, and
\begin{multline}
    \label{eq:b11}
    b^\pm_{(1,1)}(k) = b^\pm_{(1)}(k)^2
    -\gamma_0^2-2\gamma_1
    -\sum_{p}
    \left(
        \frac{p}{(p-1)^2}+f_1(p)^2-f_2(p)
    \right) \log(p)^2.
\end{multline}

In Section~\ref{sec:reform-probl} we derive
formula~\eqref{eq:146} for the leading coefficient of $Q_{\pm}(k,x)$. Our tools
are then applied, in Section~\ref{sec:further-lower-order}, to the general term
$c_\pm(r,k)$, where we obtain a formula for $N_\lambda(k)$ expressed as a sum
of determinants of the form:
\begin{equation}
  \label{eq:i1}
  D_\lambda(k)=
  \det \left( \binom{2k-i-\lambda_{k-i+1}}{2k-2j}\right)_{1\leq i,j\leq k},
\end{equation}
where $\lambda=(\lambda_1,\dotsc,\lambda_m)$ is a partition with length
$l(\lambda)\leq k$ (see Section~\ref{sec:symmetric} for definitions).

In Section~\ref{sec:determinant} we derive some
interesting formulas for these determinants. To describe our formulas,
let $y=(y_1,\dotsc,y_m)$.
We define the {\em coefficient operator} $[y^\beta]$ on the set of formal
multivariate Taylor or Laurent series in $y$, which picks the coefficient of
the monomial $y^\beta$ in the series. More precisely, if
\begin{equation}
    \label{eq:i3}
    f(y_1,\dotsc,y_m)=\sum_{r_1,\dotsc,r_m \in \Z}
    a_{r_1,\dotsc,r_m}y_1^{r_1}\dots y_m^{r_m},
\end{equation}
 define 
\begin{equation}
   \label{eq:i55}
[y_1^{u_1}\dots
y_m^{u_m}]f   =a_{u_1,\dotsc,u_m}.
\end{equation}

We prove the following Theorem.

\begin{theorem}
\label{thm:i2} Let $\lambda$ be a partition and $\mu$ be the
conjugate partition. Let $m=l(\lambda)$, and $n=l(\mu) = \lambda_1$.
For $k\geq \max(l(\lambda),\lambda_1)$, we have
\begin{multline}
    \label{eq:i4}
    D_\lambda(k)= 2^{\binom{k}{2} -|\lambda|}\\
    \times [y_1^{\lambda_1+m-1}\dots
    y_m^{\lambda_m}]\left( \prod_{1\leq i < j \leq m} (y_i
    -y_j)(1-y_i-y_j) \prod_{l=1}^m(1-y_l)^{-k-m}
    \right),
\end{multline}
and also
\begin{multline}
    \label{eq:i6}
    D_\lambda(k)= 2^{\binom{k}{2} -|\lambda|}\\
    \times [z_1^{\mu_1+n-1}\dots z_n^{\mu_n}]
    \left( \prod_{1\leq i < j \leq n} (z_i -z_j)(1+z_i+z_j) \prod_{l=1}^n(1+2z_l)(1+z_l)^{k-n}
    \right).
\end{multline}
\end{theorem}

\begin{corollary}
\label{cor:1}
Let $\lambda$ be a partition, with $l(\lambda)=m$. There is a polynomial $P_\lambda(k)$,
integer valued at integers,
of degree $|\lambda|$ such that for $k\geq \max(l(\lambda),\lambda_1)$,
\begin{equation}
    \label{eq:ii1}
    D_\lambda(k)= 2^{\binom k 2 - |\lambda|} \times \ P_\lambda(k).
\end{equation}
The leading coefficient of $P_\lambda(k)$ is
\begin{equation}
    \label{eq:2}
    \frac{
        \prod_{1\leq i < j \leq m } (\lambda_i-\lambda_j-i+j)
    }
    {
        \prod_{1\leq i \leq m }(\lambda_i + m - i)!
    }
    = \chi^\lambda(1)/|\lambda|!,
\end{equation}
where $\chi^\lambda(1)$ is the degree of the irreducible representation of the symmetric group
$S_{|\lambda|}$ indexed by $\lambda$.
In particular,
\begin{equation}
    \label{eq:ii2}
    D_0(k)= 2^{\binom k 2}.
\end{equation}
where $D_0(k)$ is the determinant associated to the empty partition.

\end{corollary}

Table \ref{tab:plambda} gives a list of the polynomials $P_\lambda(k)$
for partitions up to weight $7$. Observe, in the table, that
$P_\lambda(k)$  often has many linear factors. This fact plays a role in our formula
for $N_\lambda(k)$ so we encode it in the following corollary.

\begin{corollary}
  \label{cor:2}
  Let $\lambda$ be a partition. Then $P_\lambda(k)$ is divisible by
  \begin{equation}
      \label{eq:divisibility}
      (k-\lambda_1)(k-\lambda_1-1)\ldots(k-l(\lambda)+1)
      \times (k + \lambda_1)(k + \lambda_1 - 1) \ldots (k + l( \lambda )),
  \end{equation}
  where we take the first product to be 1 if $\lambda_1\geq l(\lambda)$,
  and the second product to be 1 if $\lambda_1 < l( \lambda )$.
\end{corollary}

Finally, in Section~\ref{sec:elliptic} we discuss the application of our
techniques to the related problem of the moments of the $L$-functions associated
to quadratic twists of an elliptic curve.

\subsection{Symmetric function theory}
\label{sec:symmetric}

We collect here some definitions and results from the theory of symmetric functions
that we use in our paper.
The details can be found in \cite[Chapter 1]{macdonald_symmetric_1995}. We have
used the notations of \cite{macdonald_symmetric_1995}.

A {\em partition} $\lambda$ is a sequence of non negative integers
$(\lambda_1,\lambda_2,\dotsc)$ such that
\begin{equation} \label{eq:i30}
    \lambda_1\geq \lambda_2 \geq \cdots,
\end{equation} and only finitely many $\lambda_i$s are non zero.

The {\em length} of the partition $\lambda$ is defined to be the number of non zero
$\lambda_i$s. We denote it by $l(\lambda)$.  The {\em weight} of a
partition $\lambda$, denoted by  $|\lambda|$ is
\begin{equation}
  \label{eq:i31}
  |\lambda|=\sum_{i\geq 1} \lambda_i.
\end{equation}
The {\em diagram} of a partition is the set of points
\begin{equation}
  \label{eq:i32}
  \{(i,j) \, | \, 1\leq i \leq l(\lambda) ,\ 1\leq j \leq \lambda_i\}.
\end{equation}
The \emph{conjugate partition} $\lambda'$ of a partition $\lambda$ is the
partition  whose diagram is
\begin{equation}
  \label{eq:i33}
  \{(i,j)\, |\, (j,i) \textrm{ is in the diagram of } \lambda\}.
\end{equation}
Equivalently, the conjugate partition  of $\lambda$ is a partition
$\lambda'=(\lambda_1',\lambda_2',\dotsc)$ where
\begin{equation}
  \label{eq:i34}
  \lambda_i' = \# \{ \lambda_j\, |\, \lambda_j \geq i\}.
\end{equation}

The symmetric group $S_n$ acts on the polynomial ring $\mathbb
Z[x_1,\dotsc,x_n]$ by permuting the independent variables
$x_1,\dotsc,x_n$. The \emph{ring of symmetric polynomials} 
in $n$-variables, $\Lambda_n$, is the set of polynomials in $\mathbb Z[x_1,\dotsc,x_n]$ which are
invariant under this action of $S_n$. The ring $\Lambda_n$ is  a
graded ring:
\begin{equation}
  \label{eq:i35}
  \Lambda_n = \bigoplus_{k\geq 0}  \Lambda^k_n,
\end{equation}
where $\Lambda_n^k$ is the set of homogeneous symmetric polynomials of
degree $k$.

For $m>n$, there is a ring homomorphism
\begin{equation}
  \label{eq:i36}
  \rho_{m,n}: \mathbb Z[x_1,\dotsc,x_m] \to \mathbb Z[x_1,\dotsc,x_n],
\end{equation}
where $\rho_{m,n}(x_i)=x_i$ for $i\leq n$, and $\rho_{m,n}(x_i)=0$ for
$i>n$. This restricts to a map
\begin{equation}
  \label{eq:i37}
\rho_{m,n}:  \Lambda^k_m \to \Lambda_n^k.
\end{equation}
The maps given by \eqref{eq:i37} define an inverse system. Let
\begin{equation}
  \label{eq:i38}
  \Lambda^k = \varprojlim \Lambda^k_n,
\end{equation}
and
\begin{equation}
  \label{eq:i39}
  \Lambda=\bigoplus_{k\geq 0} \Lambda^k.
\end{equation}
The ring $\Lambda$ is called the \emph{ring of symmetric functions}. This is
a graded ring. The definition of $\Lambda$ gives us maps
\begin{equation}
  \label{eq:i40}
  \rho_n: \Lambda \to \Lambda_n.
\end{equation}

In this paper, we shall use four $\mathbb Z$-bases, parametrized by
partitions, of the ring $\Lambda$: the monomial symmetric functions
($m_\lambda$), elementary  symmetric functions ($e_\lambda$), complete
symmetric functions ($h_\lambda$) and the Schur symmetric functions
($s_\lambda$). In addition, we shall be using power symmetric
functions ($p_\lambda$). The power symmetric functions form a $\mathbb
Q$ basis of $\Lambda\otimes_{\mathbb Z}\mathbb Q$. We shall use the
same symbols to denote their image under $\rho_n$ in $\Lambda_n$.

Given $\alpha=(\alpha_1,\dotsc,\alpha_n)$, we write $x^\alpha$ to
denote $x_1^{\alpha_1}\cdots x_n^{\alpha_n}$. Let
$\lambda$ be a partition of length less than or equal to $n$. We
define the \emph{monomial symmetric function} $m_\lambda$ by its image
under $\rho_n$ for every $n$. If $n\geq l(\lambda)$, then
\begin{equation}
  \label{eq:i41}
  m_\lambda(x_1,\dotsc,x_n)= \sum_\alpha x^\alpha,
\end{equation}
where the $\alpha$ ranges  over distinct permutations of
$(\lambda_1,\dotsc,\lambda_n)$. If $l(\lambda)>n$, then
$m_\lambda(x_1,\dotsc,x_n)=0$. For the only partition of $0$, the empty
partition, we define $m_0=1$.

Let $r\geq 0$ be an integer. The \emph{elementary symmetric function} $e_r
\in \Lambda$ is given by
\begin{equation}
  \label{eq:i42}
  e_r=\sum_{1\leq i_1< i_2<\cdots<i_r} x_{i_1}\dots x_{i_r} = m_{(1,\ldots,1)},
\end{equation}
and $e_0=1$.
For a partition $\lambda$, we define
\begin{equation}
  \label{eq:i43}
  e_\lambda=e_{\lambda_1}e_{\lambda_2}\dots.
\end{equation}
The generating function for $e_r$ is
\begin{equation}
  \label{eq:i44}
  E(t)=\sum_{r\geq 0} e_rt^r = \prod_{i\geq 1}(1+x_i t).
\end{equation}

Let $r\geq 0$ be an integer. The \emph{complete symmetric function}
$h_r$ is defined to be
\begin{equation}
  \label{eq:i45}
  h_r=\sum_{|\lambda|=r} m_\lambda.
\end{equation}
Given a partition $\lambda$, we define
\begin{equation}
  \label{eq:i46}
  h_\lambda=h_{\lambda_1}h_{\lambda_2}\dots.
\end{equation}
The generating function for $h_r$ is 
\begin{equation}
  \label{eq:i47}
  H(t)=\sum_{r\geq 0} h_r t^r = \prod_{i\geq 1} (1-x_it)^{-1}.
\end{equation}

Equations \eqref{eq:i44} and \eqref{eq:i47} give us the identity,
\begin{equation}
  \label{eq:i48}
  H(t) E(-t)=1.
\end{equation}

For $r\geq 1$, the \emph{power symmetric function} $p_r$ is defined as
\begin{equation}
  \label{eq:i49}
  p_r = \sum_{i\geq 1} x_i^r = m_{(r)}.
\end{equation}
For  a partition $\lambda$, we define
\begin{equation}
  \label{eq:i50}
  p_\lambda=p_{\lambda_1}p_{\lambda_2}\dots.
\end{equation}

Let $(\alpha_1,\dotsc,\alpha_n)\in \mathbb N^n$. We define
$a_\alpha\in \mathbb Z[x_1,\dotsc,x_n]$ by
\begin{equation}
  \label{eq:i51}
  a_\alpha(x_1,\dotsc,x_n) = \det (x_i^{\alpha_j})_{1\leq i,j \leq n}.
\end{equation}
Clearly $a_\alpha$ is skew-symmetric; that is, for $w\in S_n$,
$w(a_\alpha) = \sgn(w) a_\alpha$, where  $\sgn(w)$ is
the sign of permutation $w$. Let $\delta_n$ be the partition
\begin{equation}
    \delta_n=(n-1,n-2,\dotsc,1,0).
\end{equation}
For a partition $\lambda$ of length less than or equal to $n$, we append 0's as
necessary to $\lambda$ to create an $n$-tuple, and define
\begin{equation}
  \label{eq:i52}
  s_\lambda(x_1,\dotsc,x_n) =
  \frac{a_{\scriptscriptstyle{\delta_n\! +\!
        \lambda}}(x_1,\dotsc,x_n)}{a_{\scriptscriptstyle{\delta_n}}(x_1,\dotsc,x_n)}. 
\end{equation}
This is a polynomial. Since $s_\lambda(x_1,\dotsc,x_n)$ is a ratio of
skew-symmetric polynomials, it is a symmetric polynomial. These
symmetric polynomials are called  \emph{Schur symmetric polynomials.}
For $m>n$,  $\rho_{m,n}\left(s_\lambda(x_1,\dotsc,x_m)\right)
=s_\lambda(x_1,\dotsc,x_n)$,  hence they are represented by a
function $s_\lambda\in \Lambda$.

Using the definitions of $a_\lambda$ and $s_\lambda$, it is easy to
check that
\begin{equation}
  \label{eq:i64}
a_{\delta_n}(x_1,\dotsc,x_n)=  \prod_{1\leq i < j \leq n} (x_i -x_j),
\end{equation}
and
\begin{equation}
  \label{eq:i68}
s_{{\delta_n}}(x_1,\dotsc,x_n)=  \prod_{1\leq i < j \leq n}(x_i +x_j).
\end{equation}

Let $\lambda$ be a partition and $\lambda'$ be the conjugate
partition. Then, for $n\geq l(\lambda)$
\cite[p.41]{macdonald_symmetric_1995}
\begin{equation}
  \label{eq:i53}
  s_\lambda=\det\left(h_{\lambda_i-i+j}\right)_{1\leq i,j\leq n},
\end{equation}
and for $m \geq l(\lambda')$
\begin{equation}
  \label{eq:i54}
  s_\lambda=\det\left( e_{\lambda_i'-i+j}\right)_{1\leq i,j\leq m}.
\end{equation}
Identity \eqref{eq:i53} is called the Jacobi-Trudi identity, and
\eqref{eq:i54} is called the dual Jacobi-Trudi identity. Schur symmetric
functions satisfy  \cite[p.63]{macdonald_symmetric_1995}
\begin{equation}
  \label{eq:i11}
  \prod_{i,j\geq 1} (1-x_iy_j)^{-1} = \sum_\lambda s_\lambda(x) s_\lambda(y),
\end{equation}
and
\begin{equation}
  \label{eq:i56}
  \prod_{i,j\geq 1} (1+x_iy_j) = \sum_\lambda s_\lambda(x) s_{\lambda'}(y).
\end{equation}
The sum in \eqref{eq:i11} and \eqref{eq:i56} is over all partitions
$\lambda$. Identity \eqref{eq:i11} is called the Cauchy identity,
and \eqref{eq:i56} is called the dual Cauchy identity.

There is a \emph{fundamental involution} $\omega$, a ring
automorphism, defined on the ring of symmetric functions:
\begin{equation}
  \label{eq:i61}
  \omega(e_r)=h_r.
\end{equation}
Using \eqref{eq:i48}, we can prove that
\begin{equation}
  \label{eq:i62}
  \omega(h_r)=e_r.
\end{equation}
We also have
\begin{equation}
  \label{eq:i63}
  \omega(s_\lambda)=s_{\lambda'}, \ \textrm{and}\  \omega(p_n)=(-1)^{n-1}p_n.
\end{equation}

\section{Terms of the asymptotic expansion}
\label{sec:reform-probl}

We begin by rewriting the integrand on the right hand side of~\eqref{eq:Qkresidue}
as a ratio of a holomorphic function and a
monomial. The function $G(z_1,\dotsc,z_k)$ in \reff{eq:Gdefinition} has a
pole in each $z_j$ at $(0,\dotsc,0)$ coming from the product of the
zeta functions.
These poles are eliminated by a portion of the
Vandermonde determinants. Note that
\begin{equation}
  \label{eq:148}
    \Delta(z_1^2,\ldots,z_k^2)^2=
  \left(\prod_{1\leq i \leq j \leq k} (z_i+z_j) \right)
  \frac{\Delta(z_1,\dotsc,z_k)
    \Delta(z_1^2,\dotsc,z_k^2)}{2^k \prod_{j=1}^k z_j}. 
\end{equation}
Specifically each factor $(z_i+z_j)$ occurring here cancels a pole  coming
from $\zeta(1+z_i+z_j)$. We obtain \eqref{eq:148} by observing
\begin{align} \nonumber
    \Delta(z_1^2,\ldots,z_k^2) &=   \prod_{i<j}(z_j^2 - z_i^2)
    = \Delta(z_1,\ldots,z_k)\prod_{j>i} (z_i+z_j)\\ \label{eq:233}
    & = \Delta(z_1,\ldots,z_k)\frac{\prod_{j\geq i} (z_i+z_j)}{2^k \prod_{j=1}^k z_j}.
\end{align}

Substituting \eqref{eq:148} into \reff{eq:Qkresidue}, we have
\begin{multline}
    \label{eq:Qk}
    Q_{\pm}(k,x)=
      \frac{(-1)^{k(k-1)/2}}{k!} \frac {1}{(2\pi i)^k} \oint \cdots
    \oint A_k(z_1,\ldots,z_k) \\ \prod_{j=1}^k X(\tfrac 1 2 +z_j,a)^{-\frac
      1 2} \prod_{1\leq i \leq j \leq k} (z_i+z_j) \zeta(1+z_i+z_j) \\
    \frac{\Delta(z_1,\ldots,z_k)\Delta(z_1^2,\ldots,z_k^2)}{\prod_{j=1}^k
      z_j^{2k-1} \prod_{j=1}^k z_j}\, \exp{\bigg(\frac x 2 \sum_{j=1}^k z_j\bigg)} \,
    dz_1\ldots dz_k.
\end{multline}
Now the integrand is written as a ratio of a function which is holomorphic
in a neighbourhood of $(0,\dotsc,0)$ and a monomial.



Recall that $a_k=A_k(0,\ldots,0)$.
Let $z=(z_1,\ldots,z_k)$ and \index{$m_\lambda$}$m_\lambda(z)$ be the monomial symmetric
polynomial defined in \eqref{eq:i41}.  Let
\begin{equation}
  \label{eq:powerseries}
  \sum_{i=0}^\infty \sum_{|\lambda|=i}\index{$b^\pm_\lambda(k)$}b^\pm_\lambda(k) m_\lambda(z)
\end{equation}
be the power series expansion  of
\begin{equation}
  \label{eq:145}
  \frac{1}{a_k}   A_k(z_1,\ldots,z_k) \\ \prod_{j=1}^k X(\tfrac 1 2 +z_j,a)^{-\frac
    1 2} \prod_{1\leq i \leq j \leq k} (z_i+z_j) \zeta(1+z_i+z_j).
\end{equation}
Here, the coefficients $b^+_{\lambda}$ are associated to the $a=1$ case, and
$b^-_{\lambda}$ with $a=-1$. 

In \eqref{eq:powerseries}, the sum is over all partitions
$\lambda_1+\ldots+\lambda_k=i$, with
$\lambda_1\geq \lambda_2\geq \ldots \lambda_k \geq 0$.
We divide the expression by $a_k$ to ensure that the constant term in
the power series is $1$. We shall calculate the Taylor series of
\eqref{eq:145} by calculating the Taylor series of its logarithm. This
calculation is simpler if the constant term is  $1$, i.e.  $b^\pm_{\bf 0}(k)=1$ in
\eqref{eq:powerseries}.
So \reff{eq:Qk} becomes
\begin{multline}\label{eq:sumofintegrals}
    Q_{\pm}(k,x) = \frac{(-1)^{k(k-1)/2}}{k!} \frac {a_k}{(2\pi i)^k}
    \sum_{i=0}^\infty \sum_{|\lambda|=i}b^\pm_\lambda(k)\oint \cdots\oint m_\lambda
    (z_1,\dotsc, z_k) \\ \frac{\Delta(z_1,\dotsc,z_k)\Delta(z_1^2,\dotsc,z_k^2)}
    {\prod_{j=1}^k z_j^{2k}}  \exp{\left(\frac x 2 \sum_{j=1}^k z_j\right)}\,
    dz_1\dots dz_k.
\end{multline}
Only finitely many integrals in the sum \eqref{eq:sumofintegrals}
are nonzero.  Each of the integrals in \eqref{eq:sumofintegrals} picks up the
coefficient of $z_1^{2k-1}\dots z_k^{2k-1}$ in the Taylor expansion
of the numerator of the corresponding integrand. If $\deg m_\lambda(z_1,\dotsc,z_k)+\deg
\Delta(z_1,\dotsc,z_k) + \deg \Delta(z_1^2,\dotsc,z_k^2) > \deg
(z_1^{2k-1}\dots z_k^{2k-1})$, 
that is $|\lambda| > k(k+1)/2$, 
then in the Taylor expansion of the numerator of
\eqref{eq:sumofintegrals} the coefficient of $z_1^{2k-1}\dots
z_k^{2k-1}$ is $0$.

Given $k$, and a $\lambda$ in the sum
\eqref{eq:sumofintegrals}, the coefficient of the
monomial $z_1^{2k-1}\dots z_k^{2k-1}$ in the Taylor expansion of the
numerator of the integrand is a constant, depending on $\lambda$ and $k$,
times $x^{\frac{k(k+1)}{2}-|\lambda|}$.

\subsection{The leading term}
\label{sec:leading-term}

In this section, we shall calculate the leading coefficient of $Q_{\pm}(k,x)$,
i.e. the coefficient $c_\pm(0,k)$  of $x^{\frac{k(k+1)}{2}}$. The calculation will
also provide insight into how to calculate the lower order terms of $Q_{\pm}(k,x)$. 
The leading coefficient is the same for $Q_+(k,x)$ and for $Q_-(k,x)$, and is given
in the following proposition.

\begin{proposition}
  \label{prop:3}
  The leading coefficient $c_{\pm}(0,k)$ of $Q_{\pm}(k,x)$ in \reff{eq:Qexpansion} is
\begin{equation}
    \frac{  a_k}{ 2^{k} }
    \prod_{j=0}^{k-1}\frac{(2j)!}{(k+j)!}.
\end{equation}
\end{proposition}

The leading term in ~\reff{eq:Qexpansion} corresponds to the $i=0$ term
of \eqref{eq:sumofintegrals}. In this case there is only one integral
within the inner summation sign, giving
\begin{equation}
    \label{eq:75}
    c_\pm(0,k)x^{k(k+1)/2} = \frac{(-1)^{\frac{k(k-1)}2}}{k!(2\pi i)^k}a_k \oint\cdots\oint
    \frac{\Delta(z_1,\ldots,z_k)\Delta(z_1^2,\ldots, z_k^2)}{
    \prod_{j=1}^{k}z_j^{2k}} \exp({\tfrac x 2 \sum_{j=1}^k z_j}) dz_1\ldots dz_k.
\end{equation}

Substituting  $u_j=x z_j/2$, simplifying, and then relabeling $u_j$ with $z_j$, we obtain
\begin{equation}
  \label{eq:183}
  c_\pm(0,k)x^{k(k+1)/2}= \frac{(-1)^{\frac{k(k-1)}{2}}}{k!(2\pi i)^k} a_k \left(\frac x 2
    \right)^{\frac {k(k+1)}{2}} \oint\cdots\oint
  \frac {\Delta(z_1,\ldots,z_k)\Delta(z_1^2,\ldots,z_k^2)} {\prod_{j=1}^k z_j^{2k}} \\ \exp({\sum_{j=1}^k z_j}) \,
  dz_1\cdots dz_k.
\end{equation}

The presence of the Vandermonde determinants prevents us from separating the integrals.
However, we apply the following trick to move the Vandermonde determinants outside the integral.
Introduce new variables $x_1,\dotsc,x_k$ and
consider the more general integral
\begin{align}
  I(x_1,\ldots,x_k):= 
    \frac{1}{(2\pi i)^k}
    \oint\cdots\oint
    \frac {\Delta(z_1,\ldots,z_k)\Delta(z_1^2,\ldots,z_k^2)}
    {\prod_{j=1}^k z_j^{2k}}  \label{eq:Ixdefinition}
    \exp({\sum_{j=1}^k x_j z_j}) \,
    dz_1\cdots dz_k.
\end{align}
Thus, the evaluation of $c_{\pm}(0,k)$ boils down to determining $I(1,\ldots,1)$.

Next, we introduce a partial
differential operator which will help us move the Vandermonde
determinants outside the integral.
Note that for a polynomial $P(x_1,\dotsc,x_k)$ in $k$ variables, we have
\begin{equation}
  \label{eq:185}
  P\left(\frac{\partial}{\partial x_1},\dotsc,
    \frac{\partial}{\partial x_k} \right) \exp\left( 
    \sum_{j=1}^kx_j z_j \right) = P(z_1,\dotsc,z_k) \exp\left(
    \sum_{j=1}^kx_j z_j \right).
\end{equation}
We set
\begin{equation}
    \label{eq:246}
    q(z_1,\ldots,z_k) := \Delta(z_1,\ldots,z_k)\Delta(z_1^2,\ldots,z_k^2).
\end{equation}
Then \eqref{eq:Ixdefinition} equals
\begin{equation}
    \label{eq:186}
    \frac{1}{(2\pi i)^k}
    \oint\cdots\oint
    q\left(\frac{\partial}{\partial
      x_1},\dotsc,\frac{\partial}{\partial x_k} \right)\frac{
    \exp({\sum_{j=1}^k x_j z_j}) } 
    {\prod_{j=1}^k z_j^{2k}} \,
    dz_1\cdots dz_k.
\end{equation}
Pulling the differential operator outside the integral (Leibniz's
rule) we conclude 
that \eqref{eq:186} equals
\begin{equation}
  \label{eq:187}
    q\left(\frac{\partial}{\partial x_1},\dotsc,\frac{\partial}{\partial x_k} \right)
    \frac{1}{(2\pi i)^k}
    \oint\cdots\oint
    \frac{
    \exp({\sum_{j=1}^k x_j z_j}) } 
    {\prod_{j=1}^k z_j^{2k}} \,
    dz_1\cdots dz_k.
\end{equation}
The integrand in~\eqref{eq:187} can be written as a product of integrals
in one variable, 
\begin{equation}
  \label{eq:188}
    q\left(\frac{\partial}{\partial x_1},\dotsc,\frac{\partial}{\partial x_k} \right)
    \frac{1}{(2\pi i)^k}
      \prod_{j=1}^k \oint \frac{\exp({x_jz_j})}{ z_j^{2k}} dz_j .
\end{equation}
Each integral in the above product can be
evaluated by expanding $\exp(x_jz_j)=\sum_{n=0}^\infty (x_jz_j)^n/n!$.
The coefficient of $z_j^{2k-1}$ is $x_j^{2k-1}/(2k-1)!$,
and thus \eqref{eq:188} equals
\begin{equation}
  \label{eq:189}
  q\left(\frac{\partial}{\partial x_1},\ldots,\frac{\partial}{\partial 
      x_k}\right) \prod_{i=1}^k  \frac{x_i^{2k-1}}{(2k-1)!}.
\end{equation}

We have turned our computation of $c_\pm(0,k)$ into the question of determining
the result of applying $q(\frac{\partial}{\partial
x_1},\dotsc,\frac{\partial}{\partial x_k})$ to $\prod_{i=1}^k
\frac{x_i^{2k-1}}{(2k-1)!}$, and finding the value of the resulting polynomial
at $(1,\dotsc,1)$. This calculation is done in Lemma \ref{lem:8}. The proof  of
Lemma \ref{lem:8} uses  Lemmas \ref{lem:10}, and \ref{lem:9}.

Lemma~\ref{lem:10}, a variant of Lemma 2.1 in \cite{CFKRS2}, gives a formula
for applying the differential operator $ \Delta\left(\frac{\partial^2}{\partial
x^2_1},\ldots,\frac{\partial^2}{\partial x^2_k}\right)$ to a product of
functions.

\begin{lemma}
\label{lem:10}
\begin{equation}
    \Delta\left(\frac{\partial^2}{\partial
    x_1^2},\ldots,\frac{\partial^2}{\partial x_k^2}\ \right) \prod_{i=1}^k f(x_i) =
    \left| f^{(2j-2)}(x_i) \right|_{k\times k}.
\end{equation}
\end{lemma}
\begin{proof}
Write
\begin{equation}
    \Delta\left(\frac{\partial^2}{\partial
    x^2_1},\ldots,\frac{\partial^2}{\partial x^2_k}\right) =
    \left| \frac{\partial^{2j-2}}{ x_i^{(2j-2)}} \right|_{k\times k}.
\end{equation}
Applying this to $\prod_{i=1}^k f(x_i)$, and noticing that $x_I$ only appears
in the $i$-th row of the determinant, we can move $f(x_i)$ into that row.
\end{proof}

Lemma \ref{lem:9} gives a formula for applying a product of
differentials to a determinant of functions.
\begin{lemma}  \label{lem:9}
Let $f_1(x),\dotsc,f_k(x)$ be smooth functions of one
variable. Then
\begin{equation}
    \label{eq:79}
      \frac{\partial^{n_1}}{\partial x_1^{n_1}} \dots
      \frac{\partial^{n_k}}{\partial x_k^{n_k}} 
  \begin{vmatrix}
    f_1(x_1)& \hdotsfor{2} & f_k(x_1)\\
    \vdots & \ddots & & \vdots \\
    \vdots &  &\ddots & \vdots \\
    f_1(x_k)&\hdotsfor{2}& f_k(x_k)
  \end{vmatrix} =  \begin{vmatrix}
    f_1^{(n_1)}(x_1)& \hdotsfor{2} & f_k^{(n_1)}(x_1)\\
    \vdots & \ddots & & \vdots \\
    \vdots &  &\ddots & \vdots \\
    f_1^{(n_k)}(x_k)&\hdotsfor{2}& f_k^{(n_k)}(x_k)
  \end{vmatrix}.
\end{equation}
\end{lemma}
\begin{proof}
It is easy to see if we first look at a simple case, say
$\tfrac{\partial}{\partial x_1}$ applied to the determinant on the left hand
side of \eqref{eq:79}.
\end{proof}

\begin{lemma}
  \label{lem:8}Let $q(z_1,\ldots,z_k)= \Delta(z_1,\ldots,z_k)
\Delta(z_1^2,\ldots,z_k^2)$, then
\begin{equation}
    \label{eq:155}
    q\left(\frac{\partial}{\partial x_1},\dotsc,\frac{\partial}{\partial
      x_k}\right) \prod_{i=1}^k  \frac{x_j^{2k-1}}{(2k-1)!} 
\end{equation}
evaluated at $(x_1,\dotsc, x_k)=(1, \dotsc ,1)$ is
\begin{equation}
    \label{eq:190}
    (-1)^{\frac{k(k-1)}{2}} \times  k! \left(\prod_{j=0}^{k-1}
    \frac{(2j)!}{(k+j)!}\right) 2^{\frac{k(k-1)}{2}} .
\end{equation}
\end{lemma}
\begin{proof} To prove the Lemma, we relate the value of
  \eqref{eq:155} evaluated at $(x_1,\dotsc, x_k)=(1,\dotsc,1)$ to a
  determinant of a   matrix  whose entries are binomial
  coefficients. We then use an identity for binomial coefficients to
  rewrite the determinant as a product of two determinants, and
  evaluate each of them separately.

Applying Lemma \ref{lem:10}, we can deduce that
\begin{align}
  \Delta\left(\frac{\partial}{\partial
      x_1},\ldots,\frac{\partial}{\partial x_k}\right)
  \Delta\left(\frac{\partial^2}{\partial^2
      x_1^2},\ldots,\frac{\partial^2}{\partial
      x_k^2}\right)\prod_{j=1}^k f(x_j)
\end{align}
equals
\begin{equation}
  \label{eq:259}
  \Delta\left(\frac{\partial}{\partial x_1},\ldots,\frac{\partial}{\partial
      x_k}\right) \left| f^{(2(j-1))}(x_i) \right|_{k\times k}.
\end{equation}
Expanding the Vandermonde determinant of partial differential
operators, we obtain 
\begin{equation}
  \label{eq:74}
   \sum_{\mu \in S_k} \sgn( \mu)  \frac{\partial^{\mu_1-1}}{\partial
     x_1^{\mu_1-1}}\dots\frac{\partial^{\mu_k-1}}{\partial x_k^{\mu_k-1}}
  \begin{vmatrix}
    f(x_1) & f^{(2)}(x_1) &\cdots & f^{(2(k-1))}(x_1) \\
    f(x_2) & f^{(2)}(x_2) &\cdots & f^{(2(k-1))}(x_2) \\
    \vdots & \vdots & \ddots & \vdots \\
    f(x_k) & f^{(2)}(x_k) &\cdots & f^{(2(k-1))}(x_k)
  \end{vmatrix},
\end{equation}
where $\mu_1,\dotsc,\mu_k$ is the image of the permutation $\mu$  of $1,\dotsc,k$. Applying
Lemma \ref{lem:9},  we can see  that \eqref{eq:74} equals 
\begin{equation}
     \label{eq:73}
   \sum_{\mu \in S_k} \sgn( \mu) 
  \begin{vmatrix}
    f^{(\mu_1 -1)}(x_1) & f^{(\mu_1 +1)}(x_1) &\cdots & f^{(\mu_1 -1 +2(k-1))}(x_1) \\
    f^{(\mu_2 -1)}(x_2) & f^{(\mu_2 +1)}(x_2) &\cdots & f^{(\mu_2 -1 +2(k-1))}(x_2) \\
    \vdots & \vdots & \ddots & \vdots \\
    f^{(\mu_k - 1)}(x_k) & f^{(\mu_k + 1)}(x_k) &\cdots & f^{(\mu_k - 1 +2(k-1))}(x_k)
  \end{vmatrix}.
\end{equation}
Let \index{$f(x)$}$f(x)=\frac{x^{2k-1}}{(2k-1)!}$.
Expression \eqref{eq:73} evaluated at $(x_1,\dotsc,x_k)=(1,\dotsc,1)$ is
\begin{equation}
  \label{eq:239}
  \sum_{\mu \in S_n} \sgn( \mu) 
  \begin{vmatrix}
    \frac{1}{(2k-\mu_1)!} &\frac{1}{(2k-\mu_1-2)!}  &\cdots & \frac{1}{(-\mu_1+2)!}\\
    \frac{1}{(2k-\mu_2)!} &\frac{1}{(2k-\mu_2-2)!}  &\cdots & \frac{1}{(-\mu_2+2)!}\\
    \vdots & \vdots & \ddots & \vdots \\ 
    \frac{1}{(2k-\mu_k)!} &\frac{1}{(2k-\mu_k-2)!}  &\cdots & \frac{1}{(-\mu_k+2)!}
  \end{vmatrix}. 
\end{equation} 
Rearranging the rows to cancel the effect of $\mu$ (this introduces another $\sgn(\mu)$ in
front of the determinant) and evaluating at
$(x_1,\dotsc, x_k)=(1,\dotsc,1)$, we get \eqref{eq:239} equals
\begin{equation}
  \label{eq:213}
    k!
  \begin{vmatrix}
    \frac {1} {(2k-1)!} & \frac {1}{(2k-3)!} &\cdots  & \frac{1}{1!} \\
    \frac {1} {(2k-2)!} & \frac {1}{(2k-4)!} &\cdots  & \frac{1}{0!} \\
    \vdots & \vdots & \ddots & \vdots \\
    \frac{1}{k!} & \frac {1}{(k-1)!} &\cdots & 0
  \end{vmatrix}.
\end{equation}
We can convert the determinant \eqref{eq:213} into a determinant of
matrices whose entries are binomial coefficients.  Multiplying
the $j^\textrm{th}$ column  by $\frac 1{(2(j-1))!} $ and the
$i^\textrm{th}$ row  by $(2k-i)!$, we see that \eqref{eq:213} equals
\begin{equation}
  \label{eq:159}
    k! \frac{0! 2! \cdots (2k-2)!}{(2k-1)! (2k-2)! \cdots k!} 
  \begin{vmatrix}
    \binom{2k-1}{0} & \binom{2k-1}{2} & \cdots & \binom{2k-1}{2k-2}\\
    \binom{2k-2}{0} & \binom{2k-2}{2} & \cdots & \binom{2k-2}{2k-2}\\
    \vdots & \vdots & \ddots & \vdots \\
    \binom{k}{0} & \binom{k}{2} &\cdots &\binom{k}{2k-2}
  \end{vmatrix}.
\end{equation}

The determinant in  \eqref{eq:159} is
\begin{equation}
  \label{eq:143}
    \left|\binom{2k-i}{2j-2} \right|_{k\times k}.
\end{equation}
In Section~\ref{sec:determinant} we study this determinant. From \eqref{eq:i1}
and Corollary~\ref{cor:1}, the determinant of this matrix equals $(-2)^{\binom
k 2}$. The extra $(-1)^{\binom k 2}$ here comes about from the fact that the
$D_0(k)$ in~\eqref{eq:i1} has its columns reversed from the above
determinant.

\end{proof}

Applying Lemma \ref{lem:8} to (\ref{eq:189}), we find that the
leading term is:
\begin{align}
  \label{eq:161}
\notag &   \frac{a_k}{k!} \left(\frac x 2 \right)^{\frac{k(k+1)}{2}} \left(k!
    \frac{0! 2!\cdots (2k-2)!}{(2k-1)!\cdots k!} \right) 2^{\frac{k(k-1)}{2}}\\
    =&  \frac{a_k}{2^k}
    \prod_{j=0}^{k-1} \frac{(2j)!}{(k+j)!}
    x^{k(k+1)/2}
\end{align}
Hence the coefficient of the leading term is
\begin{equation}
\frac{  a_k}{ 2^{k} }
  \prod_{j=0}^{k-1}\frac{(2j)!}{(k+j)!}.
\end{equation}
This proves Proposition \ref{prop:3}.

\subsection{Further lower order terms}
\label{sec:further-lower-order}

In this section we consider a general term occurring in the
sum of integrals \eqref{eq:sumofintegrals}. Let $\lambda$ be
a partition. We shall calculate
\begin{multline}
    \label{eq:194}
    \frac{(-1)^{k(k-1)/2} 2^k}{k!} \frac {a_k}{(2\pi i)^k}
    b^\pm_\lambda(k)\oint \cdots\oint
    m_\lambda (z_1,\dotsc, z_k) \\
    \frac{\Delta(z_1,\dotsc,z_k)\Delta(z_1^2,\dotsc,z_k^2)}
    {2^k\prod_{j=1}^k z_j^{2k}}\exp\left(\frac x 2 \sum_{j=1}^kz_j \right) dz_1\dots dz_k.
\end{multline}

Modifying the approach of the previous section to incorporate
the extra monomial $m_\lambda(z_1,\dotsc,z_k)$, we define
\begin{equation}
  q_\lambda(z_1,\dotsc , z_k) =m_\lambda (z_1,\dotsc,z_k)
  \Delta(z_1,\dotsc,z_k) \Delta(z_1^2,\dotsc,z_k^2) .
\end{equation}
Following the same steps as in the evaluation of the leading term,
expression \eqref{eq:194} becomes 
\begin{equation}\label{generalpartition}
  \frac{(-1)^{\frac{k(k-1)}2} a_k b^\pm_\lambda(k)}{k!}\left(\frac x 
    2\right)^{\frac{k(k+1)}2-|\lambda|} \left( q_\lambda\left(\frac{\partial}{\partial
        x_1},\dotsc,\frac{\partial}{\partial x_k} \right) \prod_{j=1}^k
    \frac{x_j^{2k-1}} {(2k-1)!}\right)_{\textrm{evaluated at } {x_j=1.}}
\end{equation}
This section is devoted to calculating  \eqref{generalpartition}.

Let  $f(x)={x^{2k-1}}/{(2k-1)!} $. Let
\index{$\lvert\lambda\rvert$}$|\lambda|=\sum_i \lambda_i$,
and length $\index{$l(\lambda)$}l(\lambda)$. Thus, $l(\lambda)$ is
the number of non zero elements of the  partition $\lambda$, i.e.
$\lambda_j=0$ for $j>l(\lambda)$.
Let $m_j(\lambda)$ be the number of $j$'s in the partition $\lambda$, so that
$|\lambda| = m_1(\lambda)+2m_2(\lambda)+3m_3(\lambda)+\ldots$.  
There are
$\binom{k}{l(\lambda)}\binom{l(\lambda)}{m_1(\lambda),m_2(\lambda),\dotsc}$
monomials in $m_\lambda(x_1,\dotsc,x_k)$ (\!\!\cite{Sta99}, 7.8). Here
\index{$\binom{l(\lambda)}{m_1(\lambda),m_2(\lambda),\dotsc}$}
$\binom{l(\lambda)}{m_1(\lambda),m_2(\lambda),\dotsc}$
is the multinomial coefficient. Since we are working with symmetric functions,
it is enough to compute \eqref{eq:194}, i.e.
\eqref{generalpartition}, for one
monomial  of $m_\lambda\left(\frac{\partial}{\partial 
  x_1},\dots,\frac{\partial}{\partial x_k}\right)$.  Therefore,
\begin{equation}\label{eq:72}
  q_\lambda\left(\frac{\partial}{\partial x_1},\dotsc,\frac{\partial}{\partial
      x_k} \right) \prod_{j=1}^k f(x_j)
  \Bigg\vert_{\textrm{ evaluated at } x_j=1}
  \end{equation}
equals   
\begin{equation}
  \label{eq:254}
\binom{k}{l(\lambda)}
    {l(\lambda) \choose m_1(\lambda),m_2(\lambda),\dots} \frac{\partial^{|\lambda|}}{\partial
    x_1^{\lambda_1}\dots\partial x_{l(\lambda)}^{\lambda_{l(\lambda)}}}\Delta\left(
    \frac{\partial }{\partial x_1} \dots 
    \frac{\partial }{\partial x_k} \right)\Delta \left(
    \frac{\partial^2 }{\partial x_1^2} \dots 
    \frac{\partial^2 }{\partial x_k^2} \right)\prod_{j=1}^k f(x_j)
\end{equation}
evaluated at $(x_1,\dotsc,x_k)=(1,\dotsc,1)$. We already have the
expression for the effect of Vandermonde determinant operators in
\eqref{eq:73}. Therefore by Lemma~\ref{lem:9},  the expression
\eqref{eq:254} equals 
\begin{equation}
    \label{eq:156}\binom{k}{l(\lambda)}
   {l(\lambda) \choose m_1(\lambda),m_2(\lambda),\dots} \frac{\partial^l}{\partial
   x_1^{\lambda_1}\dots\partial x_{l(\lambda)}^{\lambda_{l(\lambda)}}} \sum_{\mu \in S_k}
  \sgn(\mu) \det \left(f^{(\mu_i -1) +2(j-1))}(x_i) \right).
\end{equation}
The expression \eqref{eq:156} is equal to
\begin{equation}
  \label{eq:201}
\binom{k}{l(\lambda)}
   {l(\lambda) \choose m_1(\lambda),m_2(\lambda),\dots}
  \sum_{\mu \in S_k} \sgn(\mu) \det \left(f^{(\mu_i -1 +2(j-1)+
      \lambda_i)}(1) \right).
\end{equation}
 In each  summand  of \eqref{eq:201}, rearrange the rows so as to
 reverse the effect of $\mu$. We get
\begin{equation}
  \label{eq:157}  \binom{k}{l(\lambda)} {l(\lambda) \choose m_1(\lambda),m_2(\lambda),\dots}
  \sum_{\nu \in S_k}  \det \left(f^{(i -1 +2(j-1)+
      \lambda_{\nu_i})}(1) \right).
\end{equation}
Here $\nu$ is $\mu^{-1}$. The expression \eqref{eq:157} is
\begin{equation}
  \label{eq:160}\binom{k}{l(\lambda)} {l(\lambda) \choose m_1(\lambda),m_2(\lambda),\dots}
\sum_{\nu \in S_k} \det \left(
    \frac{1}{\left(2k-1 -(i-1)-2(j-1)- \lambda_{\nu_i
        }\right)!} \right)_{ij}.
\end{equation}
Each determinant inside the sum is of the form 
\begin{equation}
  \label{eq:165}
  \det \left(     \frac{1}{\left(2k-1 -(i-1)-2(j-1)- d_i \right)!}
  \right)_{ij},
\end{equation}
and $\sum d_i =|\lambda|$. In Proposition~\ref{prop:5}, we determine a
necessary condition for the determinant \eqref{eq:165} to be non
zero. This condition will imply that a large portion of terms in
\eqref{eq:160} are zero.

\begin{prop}
  \label{prop:5}
Consider the determinant
\begin{equation}
  \label{eq:231}
  \det \left(     \frac{1}{\left(2k-1 -(i-1)-2(j-1)- d_i \right)!}
  \right)_{ij}.
\end{equation}
Assume that $\sum_{i=1}^k d_i =|\lambda|$, with $d_i \in \Z_{\geq 0}$.
The determinant \eqref{eq:231} is zero if any  of 
$d_1,\dotsc,d_{k-|\lambda|}$ is non zero.
\end{prop} 
\begin{proof} Let $u$ be a  number between $1$ and $k$
such that $d_u$ is non zero.    The $u^\textrm{th}$ row in the matrix is
\begin{equation}
  \label{eq:166}
  \left(
    \frac{1}{\left(2k-1 -(u-1)-2(j-1)- d_u \right)!} \right)_{1\leq j
    \leq k}.
\end{equation}
Now look at the row which is $d_u$ rows below the row $u$ in the matrix
\eqref{eq:231}. Let this be row $v$ where  $v=u+d_u$.  Row  $v$,
\begin{equation}
  \label{eq:167}
    \left(
    \frac{1}{\left(2k-1 -(v-1)-2(j-1)- d_v \right)!} \right)_{1\leq j
    \leq k},
\end{equation}
is identical to row $u$ if $d_v$ is  zero.  We have a 
necessary condition for the matrix to have a non zero determinant;
for every $u$ such that $d_u\neq 0$, either $d_{u+d_u}$ is also non zero or $
u+d_u >k$. We look at this cascading process, and see that if we start at a
row above the row $k-|\lambda|$, that is if $d_u \neq 0$ for some $u\leq
k-|\lambda|$, then we cannot go  down beyond row $k$ since all $d_i$ add to
$|\lambda|$. Hence we will have two identical rows. We can then conclude that
we obtain non zero determinants in \eqref{eq:231} only when $d_u =0$ for
$1\leq u \leq  k-|\lambda|$.
\end{proof}

The above proposition tells us that, in~\eqref{eq:76} all the action takes
place in the last $|\lambda|$ rows or lower. Thus, let 
${\bf u}=(u_k,\dotsc,u_1) $ be a permutation of $(\lambda_1,\dotsc,\lambda_k)$.
Notice that we have reversed the order of the subscripts on the $u$'s, starting
at $k$ and ending at $1$. Applying the above proposition, we shall assume
$u_i=0$ for $i>|\lambda|$. Note, however, that some of the $u_i$'s, with $i\leq
|\lambda|$ can also equal 0. For a given permutation {\bf u}, let $i({\bf u})$ be the
smallest positive integer such that $u_i=0$ for all $i>i({\bf u})$. Thus,
$i({\bf u}) \leq |\lambda|$.

Next, any two permutations that have identical non-zero $u_i$'s, i.e.
that move around the $0$'s, produce the same determinant. For any given way of selecting
where the non-zero $\lambda_i$'s go, there are $(k-l(\lambda))!$ ways to move around the
remaining zero-valued $\lambda_i$'s. Furthermore, permuting identical
non-zero $\lambda_i$'s also produces the same determinant. For a given permutation,
there are $m_1(\lambda)! m_2(\lambda)! \ldots$ ways to move around the identical non-zero
$\lambda_i$'s. Using the fact that
\begin{equation}
    \label{eq:simplify multinomial}
    \binom{k}{l(\lambda)}{l(\lambda)\choose m_1(\lambda),m_2(\lambda),\dotsc}
    (k-l(\lambda))! m_1(\lambda)! m_2(\lambda)! \ldots = k!,
\end{equation}
and taking into account the above two paragraphs, expression
\eqref{eq:160} can thus be written as
\begin{equation}\label{eq:76}
    k!
    \sum_{\bf u}{^{'}} \left|
    \begin{array}{cccccc}
      \frac{1}{(2k-1)!} & \frac{1}{(2k-3)!} &\hdotsfor{3} & \frac{1}{1!} \\
      \frac{1}{(2k-2)!} & \frac{1}{(2k-4)!} & \hdotsfor{3} & \frac{1}{0!}\\
      \vdots            & \vdots            &   & & & \vdots \\
      \frac{1}{(k+i({\bf u}))!} & \frac{1}{(k-i({\bf u})-2)!}       && &        &\vdots\\
      &      && &        & \\
      \hdashline\\
      \frac{1}{(k+i({\bf u})-1-u_{i({\bf u})})!} & \frac{1}{(k+i({\bf u})-3-u_{i({\bf u})})!}  &     & &        &\vdots \\
      \vdots            & \vdots            &   &  && \vdots \\
      \frac{1}{(k-u_1)!} & \frac{1}{(k-u_1-2)!} &\hdotsfor{3}&0
    \end{array}
    \right|.
\end{equation}
There are $k-i({\bf u})$ rows above the
horizontal dashed line and $i({\bf u})$ rows below the dotted line. The sum is
over {\em distinct} permutations $(u_k,\dotsc,u_1)$ of
$(\lambda_1,\dotsc,\lambda_k)$, satisfying $u_i=0$ for $i>|\lambda|$. Note
that, in order for a given permutation $\bf u$ to appear in the sum, we require
that $k \geq i({\bf u})$.

We may also reduce the number of terms in the sum by excluding matrices where
two or more rows of the matrix are identical. The $'$ on the sum indicates that
such terms have been excluded from the sum.

Now consider one specific term in the sum \eqref{eq:76}. As in the calculation
of the leading coefficient, multiply its $i^{\mathrm{th}}$ row by $(2k-i)!$ and
its $j^{\mathrm{th}}$ column by $1/(2(j-1))!$. This enables us to write
the determinant in  a term of \eqref{eq:76} as a product of a known quantity
and a determinant of binomial coefficients,
\begin{multline}\label{eq:78}
  \frac{\prod_{j=1}^k (2(j-1))! }{\prod_{i=1}^k (2k-i)!}
  \times
  \begin{vmatrix}
    {2k-1 \choose 0} &  {2k-1\choose 2} & \hdotsfor{2} & {2k-1\choose 2k-2} \\
    {2k-2\choose 0} &   {2k-2\choose 2} & \hdotsfor{2} &  {2k-2\choose 2k-2} \\
    \vdots &    \vdots   &  \ddots      &    & \vdots \\
    {k+i({\bf u}) \choose 0} & {k+i({\bf u}) \choose 2}  &     & \ddots &  {k+i({\bf u})\choose 2k-2} \\
    {k+i({\bf u})-1-u_{i({\bf u})} \choose 0} & {k+i({\bf u})-1-u_{i({\bf u})} \choose 2} & \hdotsfor{2}&
    {k+i({\bf u})-1-u_{i({\bf u})}\choose 2k-2}\\
    \vdots &    \vdots   &  \ddots      &    & \vdots \\
    {k-u_1\choose 0} & {k-u_1\choose 2} &\hdotsfor{2} & {k-u_1\choose
      2k-2}
  \end{vmatrix}\\
  \times (k+i({\bf u})-1)_{u_{i({\bf u})}} (k+i({\bf u})-2)_{u_{i({\bf u})-1}}\dotsm(k)_{u_1}.
\end{multline}
Here \index{$(x)_n$}$(x)_n$ is the falling factorial
$x(x-1)\dots(x-n+1)$. The last factor, the product of falling
factorials, is a polynomial of degree $|\lambda|$ in $k$. Expressions~\eqref{eq:76}
and~\eqref{eq:78} for the lower terms are the analogue of \eqref{eq:159} for
the leading coefficient. The difference is the presence of
$u_1,\dotsc,u_{i({\bf u})}$ in the determinant and the appearance of the product
of falling factorials. The latter are accounted for by the fact that the
$(2k-i)!$ is not entirely cancelled by the numerator of the binomial
coefficients in the last $i({\bf u})$ rows.

We study the above determinant in the next section. To apply the formulas of
that section, we require $u_1 \geq u_2 \geq \ldots$, which does not typically
hold for the terms in the sum of~\eqref{eq:76}. However, by swapping adjacent
rows, we can arrange that these inequalities hold. More precisely, say that
$u_m < u_{m+1}$. We can assume that, in fact, $u_m +2 \leq u_{m+1}$ since if
$u_m+1 = u_{m+1}$ then the $m$-th and $m+1$-st rows from the bottom of the
matrix coincide, and such terms are excluded from~\eqref{eq:76} since the
determinant in such cases is $0$.

Consider what happens when we swap the $m$-th and $m+1$-st rows from the
bottom. The binomial coefficient $k+m-1-u_m \choose 2j-2$ gets switched with
$k+m-u_{m+1} \choose 2j-2$ at a cost of a sign change to the determinant. On
the other hand, the new determinant is of the same form, but
with $\bf u$ replaced by $\bf u'$,
where $u'_j=u_j$ for all $j$, except for $u'_m = u_{m+1}-1$ and $u'_{m+1} =
u_m+1$. Thus we have reversed the inequality, i.e. $u'_m \geq u'_{m+1}$. Notice also
that this swapping also satisfies $\sum u'_j = \sum u_j = |\lambda|$.

Therefore, continuing in this fashion, any given determinant in the sum
in~\eqref{eq:78} is equal, up to a power of $-1$, to the same
kind of determinant but with ${\bf u}$ replaced
by, say, $\alpha({\bf u})$,
where $\alpha$ is a partition of $|\lambda|$, i.e. with $\alpha_1 \geq \alpha_2
\geq \ldots \geq 0$. Let the power of $-1$ introduced by
the row swaps that take ${\bf u}$ to $\alpha({\bf u})$ be denoted by $n({\bf
a})$. Thus, a given determinant in~\eqref{eq:78} is equal, on performing the
row swaps, to
\begin{equation}
    (-1)^{n({\bf u})+ {k \choose 2}} D_{\alpha({\bf u})}(k),
    \label{eq:one term}
\end{equation}
where $D$ is the determinant defined in~\eqref{eq:i1}. The extra $(-1)^{\binom
k 2}$ arises because the columns of $D$ in~\eqref{eq:i1} are in the reverse
order from the determinants in~\eqref{eq:78}.

Therefore,
returning to equations~\eqref{eq:sumofintegrals},~\eqref{generalpartition},
and~\eqref{eq:194}, we have, on simplifying, that the coefficient $c_{\pm}(r,k)$ of
$x^{\frac{k(k+1)}2-r}$ in $Q_{\pm}(k,x)$ can be expressed as:
\begin{multline}
    \label{eq:formula for c(r,k)}
        \frac{a_k}{2^{\frac{k(k+1)}2-r}}
        \prod_{j=0}^{k-1}\frac{(2j)!}{(k+j)!}
    \\
    \times
    \sum_{|\lambda|=r}
    b^\pm_\lambda(k)
    \sum_{\bf u}{^{'}} (-1)^{n({\bf u})} D_{\alpha({\bf u})}(k)
  \times (k+i({\bf u})-1)_{u_{i({\bf u})}} (k+i({\bf u})-2)_{u_{i({\bf u})-1}}\dotsm(k)_{u_1}.
\end{multline}

In Theorem~\ref{thm:i2} and Corollary~\ref{cor:1} we show that, for $k\geq
\max(l(\alpha), \alpha_1)$,
\begin{equation}
    \label{eq: D alpha}
    D_\alpha(k) = 2^{ {k\choose 2}-r} P_{\alpha}(k),
\end{equation}
where $P_{\alpha}(k)$ is a polynomial in $k$ of degree $|\alpha|$. Theorem~\ref{thm:i2}
also gives a formula for determining the polynomials $P$.
Hence
\begin{equation}
    \label{eq:formula for c(r,k) b}
    c_\pm(r,k) = 
    \left(
        \frac{a_k}{2^k}
        \prod_{j=0}^{k-1}\frac{(2j)!}{(k+j)!}
    \right) \sum_{|\lambda|=r}
    b^\pm_\lambda(k) N_\lambda(k),
\end{equation}
where
\begin{equation}
    \label{eq:N}
    N_\lambda(k) =
    \sum_{\bf u}{^{'}} (-1)^{n({\bf u})} P_{\alpha({\bf u})}(k)
  \times (k+i({\bf u})-1)_{u_{i({\bf u})}} (k+i({\bf u})-2)_{u_{i({\bf u})-1}}\dotsm(k)_{u_1}.
\end{equation}
The sum is
over distinct permutations ${\bf u}=(u_k,\ldots,u_1)$ formed from the partition $\lambda$ by
appending $0$'s if necessary. Furthermore, we have restricted to permutations such
that $i({\bf u}) \leq |\lambda|$, and have also excluded $\bf u$ where the
corresponding matrix has any identical rows. Finally, for a given $\bf u$ to
appear in the sum we have assumed that $i({\bf u}) \leq k$, and to apply~\eqref{eq: D alpha},
we also required that $k\geq \max(l(\alpha({\bf u})), \alpha({\bf u})_1)$.

We show that the latter assumption can
be removed. First note that $i({\bf u})=l(\alpha({\bf u}))$, because our
swapping procedure that replaces a given $\bf u$ with $\bf u'$
has $i({\bf u'})=i({\bf u})$. Next, Corollary~\ref{cor:2} tells us that
$P_\alpha(k)$ vanishes for $\alpha_1 \leq k \leq l(\alpha)-1$. Furthermore,
\begin{equation}
    \label{eq:falling}
    (k+i({\bf u})-1)_{u_{i({\bf u})}} (k+i({\bf u})-2)_{u_{i({\bf u})-1}}\dotsm(k)_{u_1}
\end{equation}
vanishes if $0 \leq k < \alpha_1$ as can be seen by examining the factor associated
to $\alpha_1$: let $u_j$ be the term that, under our swapping procedure, gets swapped
down to $\alpha_1$. The corresponding falling factorial is $(k+j-1)_{u_j}$.
But $\alpha_1 = u_j - (j-1)$, because $u_j$ gets moved down $j-1$
rows to the bottom row. Therefore,
\begin{equation}
    (k+j-1)_{u_j} = (k+j-1)_{\alpha_1+j-1} = (k+j-1) (k+j-2) \dots (k-\alpha_1+1)
\end{equation}
which is divisible by
\begin{equation}
    \label{eq:divisibility part ii}
    k (k-1) \dots (k-\alpha_1+1).
\end{equation}
Thus, we have shown that
\begin{equation}
    \label{eq:N term}
    P_{\alpha({\bf u})}(k)
    \times (k+i({\bf u})-1)_{u_{i({\bf u})}} (k+i({\bf u})-2)_{u_{i({\bf u})-1}}\dotsm(k)_{u_1}
\end{equation}
vanishes for $0 \leq k < \max(l(\alpha({\bf u})),\alpha({\bf u})_1)$.
We can, therefore, ignore, in~\eqref{eq:N},
the condition that $k\geq \max(l(\alpha({\bf u})), \alpha({\bf u})_1)$, since, in including terms with
$k< \max(l(\alpha({\bf u})), \alpha({\bf u})_1)$,
the corresponding summand in~\eqref{eq:N} vanishes.

Hence, $N_\lambda(k)$ is given by a sum over a {\em fixed}, i.e. depending only on $\lambda$
but not on $k$, number of terms $\bf u$. Each term is a polynomial of degree $2|\lambda|$
in $k$, thus $N_\lambda(k)$ is a polynomial in $k$ of degree $\leq 2|\lambda|$.

This completes the proof of Theorem \ref{thm:maintheorem}.

As an example, We compute $N_{(2,1,1)}(k)$ using \eqref{eq:N}. In this case, the
sum \eqref{eq:N} is over the 12 distinct permutations of $(2,1,1,0)$. We can truncate
at 4 terms because $|\lambda|=4$, and $i({\bf u}) \leq |\lambda|$.
Of these 12 permutations, only $(2,1,1,0)$, $(0,2,1,1)$, $(1,0,2,1)$,
and $(1,1,0,2)$ give non zero determinants. The sign $(-1)^{n({\bf u})}$ is 1 for
$(2,1,1,0)$, and -1 for the rest. We have
\begin{align}
    \label{eq:3}
    &N_{(2,1,1)}(k)=P_{(2,1,1)}\, (k)_2 (k+1)_1 (k+2)_1 - P_{(1,1,1,1)}\,
    (k+1)_2 (k+2)_1 (k+3)_1 \notag \\
    \notag & {}\quad - P_{(1,1,1,1)}\,  (k)_1 (k+2)_2 (k+3)_1 - P_{(1,1,1,1)}\,  (k)_1
    (k+1)_1 (k+3)_2 \\
    =& \frac{1}{8} (k-2)(k+3)(k^2+k-4)\times k(k-1) (k+1)(k+2)
    -\frac{1}{24}(k-3)(k-2)(k-1)(k+4)\notag \\&\times(
        (k+1)k(k+2)(k+3)+k(k+2)(k+1)(k+3)+k(k+1)(k+3)(k+2)
    ) \notag \\
    = &k (k - 1) (k - 2) (k + 3) (k + 2) (k + 1)
\end{align}

Having shown that $N_\lambda(k)$ is a polynomial in $k$ of degree $\leq
2|\lambda|$, we can determine it either using formula~\eqref{eq:N} and
the formulas in Theorem~\ref{thm:i2} for the polynomials $P$,
or else by evaluating~\eqref{eq:76} for $2|\lambda|+1$ values of $k$
and applying polynomial interpolation. More specifically,
we can work back from~\eqref{eq:76} to~\eqref{eq:194}, and divide
by $\frac{a_k}{2^k} \prod_{j=0}^{k-1}\frac{(2j)!}{(k+j)!}$ to get the formula
\begin{equation}
    \label{eq:N direct}
    N_\lambda(k) =
    \left(\frac{-1}{2}\right)^{k(k-1)/2} 2^{|\lambda|}
    \prod_{j=0}^{k-1}\frac{(k+j)!}{(2j)!}
    \sum_{\bf u}{^{'}} \left|
    \begin{array}{cccccc}
      \frac{1}{(2k-1)!} & \frac{1}{(2k-3)!} &\hdotsfor{3} & \frac{1}{1!} \\
      \frac{1}{(2k-2)!} & \frac{1}{(2k-4)!} & \hdotsfor{3} & \frac{1}{0!}\\
      \vdots            & \vdots            &   & & & \vdots \\
      \frac{1}{(k+i({\bf u}))!} & \frac{1}{(k-i({\bf u})-2)!}       && &        &\vdots\\
      &      && &        & \\
      \hdashline\\
      \frac{1}{(k+i({\bf u})-1-u_{i({\bf u})})!} & \frac{1}{(k+i({\bf u})-3-u_{i({\bf u})})!}  &     & &        &\vdots \\
      \vdots            & \vdots            &   &  && \vdots \\
      \frac{1}{(k-u_1)!} & \frac{1}{(k-u_1-2)!} &\hdotsfor{3}&0
    \end{array}
    \right|.
\end{equation}
This formula can be used for a specific choice of $\lambda$ and several values of $k$ to create a table
of values of $N_\lambda(k)$ to which polynomial interpolation can be applied. Table~\ref{tab:N_lambda}
lists the polynomials $N_\lambda(k)$ for all $|\lambda|\leq 7$.

\begin{table}[h!]
\centerline{
\begin{tabular}{|c|c|c|}
\hline
$\lambda$ & $N_\lambda(k)/r_\lambda(k)$ & $r_\lambda(k)$ \\
\hline
$[1]$ & $k+1$ & $(k)_1$ \\ \hline
$[1, 1]$ & $(k+2)(k+1)$ & $(k)_2/2$ \\ \hline
$[2]$ & $0$ & $(k)_1$ \\ \hline
$[1, 1, 1]$ & $(k+3)(k+2)(k+1)$ & $(k)_3/6$ \\ \hline
$[2, 1]$ & $(k+2)(k+1)$ & $(k)_2$ \\ \hline
$[3]$ & $-(k-1)(k+2)(k+1)$ & $(k)_1$ \\ \hline
$[1, 1, 1, 1]$ & $(k+4)(k+3)(k+2)(k+1)$ & $(k)_4/24$ \\ \hline
$[2, 1, 1]$ & $2(k+3)(k+2)(k+1)$ & $(k)_3/2$ \\ \hline
$[2, 2]$ & $0$ & $(k)_2/2$ \\ \hline
$[3, 1]$ & $-(k-2)(k+3)(k+2)(k+1)$ & $(k)_2$ \\ \hline
$[4]$ & $0$ & $(k)_1$ \\ \hline
$[1, 1, 1, 1, 1]$ & $(k+5)(k+4)(k+3)(k+2)(k+1)$ & $(k)_5/120$ \\ \hline
$[2, 1, 1, 1]$ & $3(k+4)(k+3)(k+2)(k+1)$ & $(k)_4/6$ \\ \hline
$[2, 2, 1]$ & $4(k+3)(k+2)(k+1)$ & $(k)_3/2$ \\ \hline
$[3, 1, 1]$ & $-(k-3)(k+4)(k+3)(k+2)(k+1)$ & $(k)_3/2$ \\ \hline
$[3, 2]$ & $-2(k-2)(k+3)(k+2)(k+1)$ & $(k)_2$ \\ \hline
$[4, 1]$ & $-2(k-2)(k+3)(k+2)(k+1)$ & $(k)_2$ \\ \hline
$[5]$ & $2(k-1)(k-2)(k+3)(k+2)(k+1)$ & $(k)_1$ \\ \hline
$[1, 1, 1, 1, 1, 1]$ & $(k+6)(k+5)(k+4)(k+3)(k+2)(k+1)$ & $(k)_6/720$ \\ \hline
$[2, 1, 1, 1, 1]$ & $4(k+5)(k+4)(k+3)(k+2)(k+1)$ & $(k)_5/24$ \\ \hline
$[2, 2, 1, 1]$ & $10(k+4)(k+3)(k+2)(k+1)$ & $(k)_4/4$ \\ \hline
$[2, 2, 2]$ & $0$ & $(k)_3/6$ \\ \hline
$[3, 1, 1, 1]$ & $-(k-4)(k+5)(k+4)(k+3)(k+2)(k+1)$ & $(k)_4/6$ \\ \hline
$[3, 2, 1]$ & $-(k+3)(k+2)(k+1)(3k^2+3k-40)$ & $(k)_3$ \\ \hline
$[3, 3]$ & $(k-2)(k-4)(k+5)(k+3)(k+2)(k+1)$ & $(k)_2/2$ \\ \hline
$[4, 1, 1]$ & $-4(k+3)(k+2)(k+1)(k^2+k-10)$ & $(k)_3/2$ \\ \hline
$[4, 2]$ & $0$ & $(k)_2$ \\ \hline
$[5, 1]$ & $2(k-2)(k+3)(k+2)(k+1)(k^2+k-10)$ & $(k)_2$ \\ \hline
$[6]$ & $0$ & $(k)_1$ \\ \hline
$[1, 1, 1, 1, 1, 1, 1]$ & $(k+7)(k+6)(k+5)(k+4)(k+3)(k+2)(k+1)$ & $(k)_7/5040$ \\ \hline
$[2, 1, 1, 1, 1, 1]$ & $5(k+6)(k+5)(k+4)(k+3)(k+2)(k+1)$ & $(k)_6/120$ \\ \hline
$[2, 2, 1, 1, 1]$ & $18(k+5)(k+4)(k+3)(k+2)(k+1)$ & $(k)_5/12$ \\ \hline
$[2, 2, 2, 1]$ & $30(k+4)(k+3)(k+2)(k+1)$ & $(k)_4/6$ \\ \hline
$[3, 1, 1, 1, 1]$ & $-(k-5)(k+6)(k+5)(k+4)(k+3)(k+2)(k+1)$ & $(k)_5/24$ \\ \hline
$[3, 2, 1, 1]$ & $-2(k+4)(k+3)(k+2)(k+1)(2k^2+2k-45)$ & $(k)_4/2$ \\ \hline
$[3, 2, 2]$ & $-10(k-3)(k+4)(k+3)(k+2)(k+1)$ & $(k)_3/2$ \\ \hline
$[3, 3, 1]$ & $(k-3)(k-5)(k+6)(k+4)(k+3)(k+2)(k+1)$ & $(k)_3/2$ \\ \hline
$[4, 1, 1, 1]$ & $-6(k+4)(k+3)(k+2)(k+1)(k^2+k-15)$ & $(k)_4/6$ \\ \hline
$[4, 2, 1]$ & $-10(k-3)(k+4)(k+3)(k+2)(k+1)$ & $(k)_3$ \\ \hline
$[4, 3]$ & $5(k-2)(k-3)(k+4)(k+3)(k+2)(k+1)$ & $(k)_2$ \\ \hline
$[5, 1, 1]$ & $2(k-3)(k+4)(k+3)(k+2)(k+1)(k^2+k-15)$ & $(k)_3/2$ \\ \hline
$[5, 2]$ & $5(k-2)(k-3)(k+4)(k+3)(k+2)(k+1)$ & $(k)_2$ \\ \hline
$[6, 1]$ & $5(k-2)(k-3)(k+4)(k+3)(k+2)(k+1)$ & $(k)_2$ \\ \hline
$[7]$ & $-5(k-1)(k-2)(k-3)(k+4)(k+3)(k+2)(k+1)$ & $(k)_1$ \\ \hline
\end{tabular}
}
\caption[$N_\lambda(k)$]
{We display the polynomials $N_\lambda(k)$, for all $|\lambda| \leq 7$.
Because each monomial of $m_\lambda(z)$ contributes the same
to~\eqref{eq:sumofintegrals}, $N_\lambda(k)$ has, as a factor, the polynomial:
$r_\lambda(k):=\binom{k}{l(\lambda)}{l(\lambda)\choose
m_1(\lambda),m_2(\lambda),\dotsc} = (k)_{l(\lambda)}/(m_1(\lambda)!
m_2(\lambda)!\ldots)$, where $(k)_m = k(k-1)\ldots(k-m+1)$. Therefore, rather
than display $N_\lambda(k)$, here we list $N_\lambda(k)/r_\lambda(k)$.
}\label{tab:N_lambda}
\end{table}

\section{The coefficients $b^\pm_\lambda(k)$}
\label{sec:b}

In order to compute the multivariate Taylor expansion of~\eqref{eq:multivariateM},
we consider the series expansion of its logarithm. We first examine the
arithmetic product, and let
\begin{equation}
    \label{eq:log A}
    \log(A_k(z_1, \ldots, z_k)/a_k)
    =: \sum_{r=1}^\infty \sum_{|\lambda| = r} B_\lambda(k) m_\lambda(z).
\end{equation}
We start the sum at $r=1$ because the division by $a_k$ makes the constant term 0.
Now, the lhs is symmetric in the $z_i$'s, and we can find $B_\lambda(k)$ by
applying
\begin{equation}
    \label{eq:diff}
    \frac{1}{\lambda_1! \lambda_2! \ldots \lambda_l!}
    \frac{\partial^{\lambda_1}}{\partial z_1^{\lambda_1}}
    \frac{\partial^{\lambda_2}}{\partial z_2^{\lambda_2}}
    \ldots
    \frac{\partial^{\lambda_l}}{\partial z_l^{\lambda_l}},
\end{equation}
where $l=l(\lambda)$,
and setting $z_1 = \ldots = z_k = 0$. Since the partial derivatives do
not involve $z_{l+1},\ldots,z_k$ we can set these to 0 before the differentiation.
Thus, by~\eqref{eq:Ak}, $B_\lambda(k)$ is equal to~\eqref{eq:diff} applied to
\begin{multline}
    \label{eq:log A l}
    -\log(a_k)
    + \sum_p
    \sum_{1\leq i\leq j \leq l} \log \left(1-\frac{1}{p^{1+z_i+z_j}} \right)
    +
    \sum_{1\leq i \leq l} (k-l) \log \left(1-\frac{1}{p^{1+z_i}} \right) \\
    +
    \log
    \left(\frac 1 2
        \left(
            \prod_{j=1}^{l}
            \left(1-\frac 1 {p^{\frac 1 2 +z_j}} \right)^{-1}
            \left(1-\frac 1 {p^{\frac 1 2 }} \right)^{l-k}
             +
            \prod_{j=1}^{l}
            \left(1+\frac 1 {p^{\frac 1 2 +z_j}} \right)^{-1}
            \left(1+\frac 1 {p^{\frac 1 2 }} \right)^{l-k}
        \right)
        +\frac 1 p
    \right)
    - \log \left(1+\frac 1 p \right),
\end{multline}
evaluated at $z_1=\ldots=z_l=0$. Likewise, we can find the coefficients of
the expansions
\begin{equation}
    \label{eq:log X}
    -\frac{1}{2} \sum_{j=1}^k \log X(1/2+z_j,a)
    =: \sum_{r=1}^\infty \sum_{|\lambda| = r} f^{\pm}_\lambda(k) m_\lambda(z),
\end{equation}
where $a=1$ for $f^-$ and $0$ for $f^+$, and of
\begin{equation}
    \label{eq:log zeta}
    \sum_{1\leq i\leq j \leq k} \log \left(\zeta(1+z_i+z_j)(z_i+z_j)\right)
    =: \sum_{r=1}^\infty \sum_{|\lambda| = r} g_\lambda(k) m_\lambda(z),
\end{equation}
by applying~\eqref{eq:diff}, at $z_1=\ldots=z_l=0$, to
\begin{equation}
    \label{eq:log X l}
    -\frac{1}{2} \sum_{j=1}^l \log X(1/2+z_j,a),
\end{equation}
and to
\begin{equation}
    \label{eq:log zeta l}
    \sum_{1\leq i\leq j \leq l} \log \left(\zeta(1+z_i+z_j)(z_i+z_j)\right)
    +\sum_{1\leq i \leq l} (k-l)\log \left(\zeta(1+z_i)z_i\right),
\end{equation}
respectively.

Next, by composing the three series expansions~\eqref{eq:log A}, \eqref{eq:log
X}, \eqref{eq:log zeta} with the series for the exponential function, we can
derive formulas for the coefficients $b^{\pm}_\lambda(k)$. Example formulas,
for $b_{(1)}^\pm(k)$ and $b_{(1,1)}^\pm(k)$, are displayed in the introduction.

To obtain numerical approximations to $b^{\pm}_\lambda(k)$ for specific choices
of $k$ and $\lambda$ one needs to compute infinite sums over primes where the
summand is a rational function of $p^{1/2}$ times $\log(p)^{|\lambda|}$. This
can be achieved to high precision using Mobius inversion as described, in the
context of the moment polynomials of the Riemann zeta function, in Section 4.1
of~\cite{CFKRS2}. In this fashion, and using~\eqref{eq:147}, we computed the
values of $c_{\pm}(r,k)$, for $r\leq 10$ and $k\leq 9$, given in
Tables~\ref{tab:cminusrk} and \ref{tab:cplusrk}.

\section{Determinant of a matrix of binomial coefficients}
\label{sec:determinant}

\begin{proof}[Proof of Theorem \ref{thm:i2}]

We shall first prove \eqref{eq:i6}, and use it to prove \eqref{eq:i4}. \\

{\em Proof of \eqref{eq:i6}}:
\\

For a $k$-tuple  $(\alpha_1,\dotsc,\alpha_k)$ and $x = (x_1, \dots ,x_k)$, let
$x^\alpha$ denote the monomial $x_1^{\alpha_1}\dots x_k^{\alpha_k}$.  For a
partition $\lambda$ of length less than or equal to $k$, $x^\lambda$ can be
defined by appending zeros after the positive elements of $\lambda$ to make it
a $k$-tuple.

Reversing the rows of the matrix in \eqref{eq:i1}, we see that
\begin{equation}
  \label{eq:i5}
    D_\lambda(k)=(-1)^{\binom k 2}\det \left( \binom{k+i-1-\lambda_{i}}{2k-2j}\right)_{1\leq i,j\leq k}.
\end{equation}
The $(i,j)$th entry of the matrix in \eqref{eq:i5} can be written using the
coefficient operator defined in \eqref{eq:i55}. Let $x=(x_1,\dotsc,x_k)$. Then
\begin{equation}
    \label{eq:i2}
    D_\lambda(k) =(-1)^{\binom k 2} \det \left(
    [x_j^{2k-2j}](1+x_j)^{k+i-1-\lambda_{i}} \right)_{1\leq i,j \leq k}.
\end{equation}
Noticing that column $j$ only involves $x_j$, we can  move all the $[x_j^{2k-2j}]$
in front of the determinant to get
\begin{align}
    \notag & (-1)^{\binom k 2} [x^{2\delta_k}] \det \left(
    (1+x_j)^{k+i-1-\lambda_{i}}\right)_{1\leq i,j \leq k} \\ &=(-1)^{\binom k
    2}  [x^{2\delta_k}] \det \left( (1+x_j)^{-(k-i+\lambda_i)} \right)
    \prod_{l=1}^k (1+x_l)^{2k-1}.
\end{align}
The determinant in \eqref{eq:i2} can be written in terms of
$a_{\delta_k}$ and $s_\lambda$ defined in \eqref{eq:i51} and \eqref{eq:i52},
\begin{multline}
  \label{eq:i9}
 D_\lambda(k)=(-1)^{\binom k 2}
[x^{2\delta_k}]  a_{\lambda+\delta_k} (\tfrac
1{1+x_1}, \dotsc, \tfrac 1 {1+x_k} ) \prod_{l=1}^k (1+x_l)^{2k-1}\\
= (-1)^{\binom k 2}[x^{2\delta_k}]  a_{\delta_k} (\tfrac
1{1+x_1}, \dotsc, \tfrac 1 {1+x_k} )s_{\lambda} (\tfrac
1{1+x_1}, \dotsc, \tfrac 1 {1+x_k} ) \prod_{l=1}^k (1+x_l)^{2k-1}.
\end{multline}
But \eqref{eq:i64} gives $a_{\delta_k}(x_1,\dotsc,x_k)$ explicitly. Hence
\begin{align}
    \label{eq:i10}\nonumber
    a_{\delta_k} (\tfrac
    1{1+x_1}, \dotsc, \tfrac 1 {1+x_k} ) &= \prod_{1\leq i < j \leq k}
    \left( \frac{1}{1+x_i} -\frac{1}{1+x_j}\right)
    =\prod_{1\leq i < j \leq k } \frac{(1+x_j) - (1+x_i)}{(1+x_j)(1+x_i)} \notag \\
    & =\frac{\prod_{1\leq i < j \leq k } (x_j-x_i)}{\prod_{j=1}^{k}
    (1+x_j)^{k-1}}
    =  \frac{(-1)^{\binom k 2} a_{\delta_k}(x)}{
      \prod_{j=1}^{k}(1+x_j)^{k-1} }.
\end{align}
Using \eqref{eq:i10} in \eqref{eq:i9},  we have
\begin{equation}
  \label{eq:i13}
  D_\lambda(k) =  [x^{2\delta_k}] a_{\delta_k}(x_1,\dotsc,x_k)
  s_\lambda(\tfrac{1}{1+x_1},\dotsc,\tfrac{1}{1+x_k})
  \prod_{l=1}^k (1+x_l)^{k}.
\end{equation}

We shall now express $s_\lambda$
as a coefficient in a polynomial which is easier to work with. The
dual Jacobi-Trudi identity,
\eqref{eq:i54}, gives
\begin{equation}
  \label{eq:i15}
  s_\lambda = \det \left(e_{\mu_i-i +j}\right)_{1\leq  i,j\leq
    n},
\end{equation}
where $(\mu_1,\ldots,\mu_n)$ is the conjugate
partition of $\lambda$, and $n=l(\mu)$.

Expanding the determinant, we get
\begin{equation}
    \label{eq:i16}
    s_\lambda = \sum_{\sigma\in S_n}\sgn(\sigma) \prod_{i=1}^n
    e_{\mu_i - i +\sigma(i)}(x).
\end{equation}
From \eqref{eq:i44}, we have
$e_r=[t^r]E(t)$. We rewrite \eqref{eq:i16} using
this notation. Let $t=(t_1,\dotsc,t_n)$. Then
\begin{align}
    \label{eq:i20}
    \nonumber
     s_\lambda(x) &= \sum_{\sigma\in S_n} \sgn(\sigma) [t_1^{\mu_1-1+\sigma(1)}\dots
    t_n^{\mu_n -n +\sigma(n)}]\prod_{i=1}^nE(t_i)\\
    \nonumber &=  [t_1^{\mu_1}\dots
    t_n^{\mu_n}]\prod_{i=1}^nE(t_i)
    \sum_{\sigma\in S_n} \left(\sgn(\sigma)
    \prod_{i=1}^n t_i^{i-\sigma(i)}\right).
\end{align}
Next, pull out $\prod t_i^i$ from the sum, and multiply and divide by $\prod t_i^n$, to get
\begin{align}
    &[t_1^{\mu_1}\dots
    t_n^{\mu_n}]\prod_{i=1}^n E(t_i)\prod_{1\leq i < j \leq n}(t_i
    -t_j) \prod_{l=1}^n t_l^{l-n} \notag \\
    \nonumber  &= [t^{\mu+\delta_n}] \prod_{i=1}^n E(t_i)  \prod_{1\leq i < j \leq n }(t_i
    -t_j)
    \\
    & = [t^{\mu+\delta_n}]a_{\delta_n}(t) \prod_{i=1}^n E(t_i).
\end{align}
Here we have also used the Vandermonde determinant (up to $(-1)^{n\choose2}$):
\begin{equation}
    \sum_{\sigma\in S_n} \sgn(\sigma)
    \prod_{i=1}^n t_i^{n-\sigma(i)} = \prod_{1\leq i < j \leq n }(t_i -t_j).
\end{equation}

We have expressed $s_\lambda(x)$ as a coefficient in a
polynomial. Substituting \eqref{eq:i20} for $s_\lambda$  in
\eqref{eq:i13}, and using
the product form of $E(t)$, \eqref{eq:i44}, we have
\begin{align}
  \label{eq:i21}
\nonumber  D_\lambda(k)& = [t^{\mu+\delta_n} x^{2\delta_k}] a_{\delta_k}(x)
  a_{\delta_n}(t)
\left[\prod_{i=1}^n \left( \prod_{l=1}^k (1+
    \tfrac{t_i}{1+x_l})\right) 
 \right]\prod_{l=1}^k (1+x_l)^k\\
\nonumber &=[t^{\mu+\delta_n} x^{2\delta_k}] a_{\delta_k}(x)  a_{\delta_n}(t) 
\prod_{l=1}^k\left( (1+x_l)^{k-n} \prod_{i=1}^n (1+x_l+t_i)\right)
\\
 & =[t^{\mu+\delta_n} x^{2\delta_k}] a_{\delta_k}(x)  a_{\delta_n}(t) 
\prod_{l=1}^k\left( (1+x_l)^{k-n} \prod_{i=1}^n
  (1+\tfrac{x_l}{1+t_i})\right) \prod_{i=1}^n(1+t_i)^k
\end{align}
Applying the dual Cauchy identity  \eqref{eq:i56} to the double product on the rhs above
gives
\begin{equation}
  \label{eq:i12}
  \prod_{l=1}^k\left( (1+x_l)^{k-n} \prod_{i=1}^n
  (1+\tfrac{x_l}{1+t_i})\right) = \sum_\lambda
s_\lambda(x_1,\dotsc,x_k)s_{\lambda'}(1,\dotsc,1,\tfrac{1}
{1+t_1},\dotsc,\tfrac{1}{1+t_n}). 
\end{equation}
The number of $1$'s in the second factor on the right hand side of
\eqref{eq:i12} is $k-n$.
Recall from \eqref{eq:i52} that $a_{\delta}s_\lambda=a_{\delta+\lambda}$. Hence \eqref{eq:i21} equals
\begin{equation}
  \label{eq:i14}
  [t^{\mu+\delta_n}][x^{2\delta_k}] a_{\delta_n}(t) \prod_{i=1}^n
  (1+t_i)^k \sum_\lambda a_{\lambda\! +\! \delta_k}(x_1,\dotsc,x_k)
  s_{\lambda'}(1,\dotsc,1, \tfrac {1}{1+t_1},\dotsc, \tfrac
  {1}{1+t_n}).
\end{equation}
The monomial $x^{2\delta_k}$ occurs in the sum in \eqref{eq:i14} only
when $\lambda=\delta_k$. The coefficient of
$x^{2\delta_k}$ in $a_{2\delta_k}(x_1,\dotsc,x_k)$ is 1. Simplifying
\eqref{eq:i14}, we have
\begin{equation}
    \label{eq:i24}
    D_\lambda(k)=
    [t^{\mu+\delta_n}] a_{\delta_n}(t) \prod_{i=1}^n (1+t_i)^k
    s_{\delta_k}(1,\dotsc,1, \tfrac {1}{1+t_1},\dotsc, \tfrac
    {1}{1+t_n}).
\end{equation}
Note that we have used $\delta_k' =\delta_k$. Applying the formula
$s_{\delta_k}$ in~\eqref{eq:i68}, we have
\begin{equation}
    \label{eq:s explicit}
    s_{\delta_k}(1,\dotsc,1, \tfrac {1}{1+t_1},\dotsc, \tfrac
      {1}{1+t_n})=
    2^{\binom{k-n}2}
    \prod_{i=1}^n\left(  1+\frac 1{1+t_i} \right)^{k-n}
    \prod_{1\leq i < j \leq n}\left( \frac
    1 {1+t_i}+ \frac{1}{1+t_j}\right).
\end{equation}
The $2^{\binom{k-n}2}$ comes from pairing, in applying~\eqref{eq:i68}, the $k-n$ 1's. The middle factor arises
from matching each $1/(1+t_i)$ with $k-n$ 1's, and the last factor from matching all pairs of
distinct $1/(1+t_i), 1/(1+t_j)$. Substituting~\eqref{eq:s explicit}
into~\eqref{eq:i24}, and collecting the powers of $(1+t_i)$ gives
\begin{equation}
    D_\lambda(k) =[t^{\mu+\delta_n}] a_{\delta_n}(t)\ 2^{\binom{k-n}{2}}
    \left(\prod_{i=1}^n(1+t_i)(2+t_i)^{k-n} \right)
    \prod_{1\leq i < j \leq n}(2+t_i+t_j).
\end{equation}
Substituting $z_i=t_i/2$, and collecting powers of $2$ (note that
${k-n \choose 2} +(k-n)n + {n \choose 2}= {k \choose 2}$), we get
\begin{multline}
  \label{eq:i29}
    D_\lambda(k)=  [z^{\mu+\delta_n}] a_{\delta_n}(z) 2^{\binom k 2 + \binom n 2 -
    |\mu+\delta_n|} \prod_{i=1}^n(1+2z_i)(1+z_i)^{k-n} \\
    \times \prod_{1\leq i < j \leq n} (1+z_i+z_j).
\end{multline}
Here we have also used $a_{\delta_n}(t) = a_{\delta_n}(z) 2^{n \choose 2}$.
Since $|\delta_n|=\binom n 2$, and $|\mu|=|\lambda|$, this proves \eqref{eq:i6}.\\

{\em Proof of \eqref{eq:i4}}:\\

We now use~\eqref{eq:i6} to prove \eqref{eq:i4}. As above, let $z=(z_1,\dotsc,z_n)$.

Since Schur symmetric functions form a $\mathbb Z$-basis for the ring
of symmetric functions, the coefficient of $s_\gamma$ of a symmetric function
$F$ is well defined. We denote this coefficient by $[s_\gamma]F$.

For a symmetric polynomial  $F(z)$ in $n$-variables, and a partition $\gamma$
with length at most $n$, we have
\begin{equation}
  \label{eq:i7}
  [s_\gamma(z)] F(z) = [z^{\gamma+\delta_n}] a_{\delta_n}(z) F(z).
\end{equation}
This can be seen by writing $F(z)$ in terms of our Schur basis
\begin{equation}
    F(z) = \sum_\gamma v_\gamma s_\gamma(z).
\end{equation}
We wish to find the coefficient $v_\gamma$.
Multiplying by $a_{\delta_n}(z)$ and using~\eqref{eq:i52} gives
\begin{equation}
    a_{\delta_n}(z) F(z) = \sum_\gamma v_\gamma a_{ \gamma + \delta_n }(z).
\end{equation}
Now, the monomials in $a_{ \gamma + \delta_n }(z)$
are all distinct, and distinct from the monomials in
$a_{\gamma' + \delta_n }$ for any different partition $\gamma'$ of length at most
$n$. Furthermore, $z^{\gamma+\delta_n}$ appears in $a_{ \gamma + \delta_n }(z)$
with coefficient $1$, coming from the main diagonal of~\eqref{eq:i51}. Thus,
$v_\gamma$ is equal to the coefficient of $z^{\gamma+\delta_n}$ in
$a_{\delta_n}(z) F(z)$.

Therefore we can rewrite \eqref{eq:i6} as
\begin{equation}
    \label{eq:i59}
    D_\lambda(k)=  2^{\binom k 2 - |\lambda|}
    [s_{\mu}(z)] \prod_{1\leq i< j \leq n}
    \left(\frac{1+z_i+z_j}{(1+z_i)(1+z_j)}\right)
    \prod_{i=1}^n (1+2z_i)(1+z_i)^{k-1}.
\end{equation}
We shall work with the ring of symmetric functions $\Lambda$ instead of the ring
of symmetric polynomials in $n$ variables $\Lambda_n$. The right hand
side of \eqref{eq:i59} equals 
\begin{equation} 
  \label{eq:i60}
  2^{\binom k 2-|\lambda|}  [s_\mu(z)] \prod_{ 1\leq i < j}
  \left(\frac{1+z_i+z_j}{(1+z_i)(1+z_j)}\right)   \prod_{i\geq 1} (1+2z_i)(1+z_i)^{k-1}.
\end{equation}
Note that in \eqref{eq:i60}, we are looking at elements in the ring of
symmetric functions, $\Lambda$, i.e. as a product involving a countable number of
variables $z_1,z_2,\ldots$,
whereas in \eqref{eq:i59}, we were
considering the elements in the ring of symmetric polynomials in $n$
variables, $\Lambda_n$.

Applying $\omega$, and using \eqref{eq:i63} we obtain
\begin{equation}
  D_\lambda(k) =2^{\binom k 2 - |\lambda|} [s_\lambda(y)]\ \omega\! \left(
\prod_{1\leq i<j} \! \left( 1 -\frac{y_iy_j}{(1+y_i)(1+y_j)}\right)\!
\prod_{i\geq 1 }(1+2y_i)(1+y_i)^{k-1}
\right)
\end{equation}
We use the fact that $\exp (\log(1+u))=1+u$ to write the argument of
$\omega$ as formal power series:
\begin{multline}
\notag
D_\lambda(k)  =2^{\binom k 2 -|\lambda|}  [s_\lambda(y)]\ \omega\! \left(
\exp\sum_{a\geq 1}\frac 1 a\left( -\sum_{1\leq i<j} y_i^ay_j^a
  (1+y_i)^{-a}(1+y_j)^{-a} \right.\right.\\ \left. \left.
- (-2)^a \sum_{i\geq 1}y_i^a -(k-1)
(-1)^a\sum_{i\geq 1} y_i^a\right) \right)
\end{multline}
\begin{multline}
  \label{eq:i67}
\phantom{D_\lambda(k)}  =2^{\binom k 2 -|\lambda|}  [s_\lambda(y)]\
\omega\! \left(
\exp\sum_{a\geq 1}\frac 1 a\left( - \sum_{b,c\geq 0}
   \binom{-a}{b}\binom{-a}{c}
   \sum_{1\leq i<j} y_i^{a+b}y_j^{a+c} \right.\right.\\ \left.\left.
- (-2)^a \sum_{i\geq 1}y_i^a -(k-1)
(-1)^a\sum_{i\geq 1} y_i^a\right)\right).
\end{multline}
We can rewrite the argument of $\omega$ in \eqref{eq:i67} using power
symmetric functions;
\begin{multline}
D_\lambda(k)  =2^{\binom k 2 -|\lambda|}  [s_\lambda(y)]\ \omega\! \left(
\exp\sum_{a\geq 0}\frac 1 a\left( -\sum_{b,c\geq 0}
  \right.\right.\\  \left.\left.  \phantom{\sum_A}
   \binom{-a}{b}\binom{-a}{c}\tfrac 1 2 (p_{a+b}p_{a+c} -p_{2a+b+c})
- (-2)^a p_a -(k-1)
 (-1)^a p_a\right)\right).
\end{multline}
In \eqref{eq:i63} we have seen that $\omega(p_a) = (-1)^{a-1}p_a$. This gives
\begin{multline}
D_\lambda(k)  =2^{\binom k 2 -|\lambda|}  [s_\lambda(y)]
\exp\sum_{a\geq 0}\frac 1 a\left( -\sum_{b,c\geq 0}
   \binom{-a}{b}\binom{-a}{c}\right. \\ \left. \phantom{\sum_A}
   (-1)^{2a+b+c}\tfrac 1 2 (p_{a+b}p_{a+c} +p_{2a+b+c}) 
+2^a p_a +(k-1)
  p_a\right)
\end{multline}
\begin{equation}
  \label{eq:i72}
\phantom{D_\lambda(k) a }  =2^{\binom k 2 -|\lambda|}
[s_\lambda(y)]\prod_{1\leq i\leq j} \left( 
    1-\frac{y_iy_j}{(1-y_i)(1-y_j)} \right) \prod_{i\geq 1}
  (1-2y_i)^{-1} (1-y_i)^{-k+1}.
\end{equation}
If we isolate factors corresponding to $i=j$ in the first product, we
are able to cancel some factors in the second product. Simplifying, we
get
\begin{equation}
  \label{eq:i75}
  D_\lambda(k) =   2^{\binom k 2 -|\lambda|}
  [s_\lambda(y)]\prod_{1\leq i< j} \left(
    1-\frac{y_iy_j}{(1-y_i)(1-y_j)} \right) \prod_{i\geq 1} (1-y_i)^{-k-1}.
\end{equation}

To calculate the coefficient of $s_\lambda$ in \eqref{eq:i75}, we only
have to look at the projection in any $\Lambda_m$ such
that $m$ is greater than or equal to $l(\lambda)$. We can choose it to
be equal to $l(\lambda)$ (also equals $\mu_1$). Let $m=l(\lambda)$. Then
\begin{equation}
    D_\lambda(k)
    =
    2^{\binom k 2 -|\lambda|}  [s_\lambda(y)]\prod_{1\leq i < j\leq m}
    \left( 1-\frac{y_iy_j}{(1-y_i)(1-y_j)} \right)
    \prod_{i=1}^m (1-y_i)^{-k-1},
\end{equation}
which is equal to
\begin{equation}
  \label{eq:i74}
  2^{\binom k 2 -|\lambda|}  [s_\lambda(y)]\prod_{1\leq i < j\leq m} 
    (1-y_i-y_j) \prod_{i=1}^m
 (1-y_i)^{-k-m}.
\end{equation}
Another application of \eqref{eq:i7} proves \eqref{eq:i4}.

\end{proof}

\begin{proof}[Proof of Corollary~\ref{cor:1}]
It is immediate from  \eqref{eq:i4} or \eqref{eq:i6} that $P_\lambda(k)$ is a
polynomial in $k$ with integer values at integers of degree at most
$|\lambda|$. We will show that it is in fact of degree $|\lambda|$ and determine
its leading coefficient.

From \eqref{eq:i4}, the highest power of $k$ occurs when we pick as many
powers of $y_i$ as possible from the last product. This happens when none
of the $y_i$ are picked from $(1-y_i-y_j)$.  Note that
\begin{equation}
    \label{eq:binomial -}
    (1-y)^{-k-m} = 1 +(k+m) y +\frac{(k+m)(k+m+1)}{2!} y^2 + \ldots
    = \sum_{j=0}^\infty (k+m+j-1)_{j}\, y^j/j!
\end{equation}
The coefficient of the highest power of $k$ that appears in
the $j$-th term of this Taylor series is $1/j!$. Thus, the coefficient of $k^{|\lambda|}$
in $P_\lambda(k)$ equals
\begin{align}
    \label{eq:leading}
    \nonumber &  [y^{\lambda_{1+m-1}}_1\dots y_m^{\lambda_m}] \left( \prod_{1\leq i <
    j \leq m}(y_i -y_j)\right) \exp(y_1+\dots+y_m)\\
    \nonumber =& [y^{\lambda_{1+m-1}}_1\dots y_m^{\lambda_m}] \sum_{\sigma\in
    S_m}\sgn(\sigma) \left(\prod_{i=1}^m y_i^{m-\sigma(i)}
    e^{y_i}\right)\\
    \nonumber
    =& \sum_{\sigma\in S_m}\sgn(\sigma) \frac{1}{(\lambda_i-i+\sigma(i))!}
    = \det(1/(\lambda_i-i+j)!)_{m\times m} \\
    =&
    \frac{
        \prod_{1\leq i < j \leq m } (\lambda_i-\lambda_j-i+j)
    }
    {
        \prod_{1\leq i \leq m }(\lambda_i + m - i)!
    } = \frac{\chi^\lambda(1)}{|\lambda|!},
\end{align}
where $\chi^\lambda(1)$ is the degree of the irreducible representation
of $S_{|\lambda|}$ indexed by $\lambda$.
See example $6$ in Chapter I.7 of~\cite{macdonald_symmetric_1995} for
the last two equalities.


\end{proof}

\begin{proof}[Proof of Corollary~\ref{cor:2}]

We use equation~\eqref{eq:i6} which gives a formula for $P_\lambda(k)$. As part
of the process of identifying the coefficient of $z_1^{\mu_1+n-1}\dots z_n^{\mu_n}$ in
that formula, we focus on the coefficient of $z_1^{\mu_1+n-1}$. Now, $\mu_1 = l(\lambda)$,
and $n=\lambda_1$, hence $\mu_1+n-1 = l(\lambda)+\lambda_1-1$.
When we expand~\eqref{eq:i6}, some of the powers of $z_1$
come from the factor $(1+z_1)^{k-\lambda_1}$, and the rest from
\begin{equation}
    \label{eq:z_1 poly}
    \prod_{1 < j \leq \lambda_1} (z_1 -z_j)(1+z_1+z_j) (1+2z_1).
\end{equation}
Consider the terms arising from taking a $z_1^j$ from the above. Notice
that~\eqref{eq:z_1 poly} is a polynomial in $z_1$ of degree $2\lambda_1-1$, and thus
$0 \leq j \leq 2\lambda_1-1$. The remaining $l(\lambda)+\lambda_1-1-j$
powers of $z_1$ come from expanding $(1+z_1)^{k-\lambda_1}$ using the binomial theorem,
so that the term associated with a particular choice of $j$ is divisible by
\begin{equation}
    {k-\lambda_1 \choose l(\lambda)+\lambda_1-1-j}
    = \frac{ (k-\lambda_1)(k-\lambda_1-1)\ldots(k-2 \lambda_1-l(\lambda)+2+j) }
    {(l(\lambda)+\lambda_1-1-j)!}.
\end{equation}
For all $0 \leq j \leq 2\lambda_1-1$, this is divisible by
\begin{equation}
    \label{eq:z_1 coeff}
    (k-\lambda_1)(k-\lambda_1-1)\ldots(k-l(\lambda)+1).
\end{equation}
The coefficient of $z_1^{l(\lambda)+\lambda_1-1}=z_1^{\mu_1+n-1}$ in the
expression in~\eqref{eq:i6} is therefore divisible by~\eqref{eq:z_1 coeff}.
Thus, so is the coefficient of $z_1^{\mu_1+n-1}\dots z_n^{\mu_n}$.

The same analysis applied to~\eqref{eq:i4}, and using~\eqref{eq:binomial -} gives
that $P_\lambda(k)$ is divisible, for $l(\lambda)\leq \lambda_1$, by
\begin{equation}
    (k + l( \lambda )) \ldots (k + \lambda_1-1)(k + \lambda_1).
\end{equation}

\end{proof}


\section{Family of quadratic twists of elliptic curve L-functions}
\label{sec:elliptic}

Here we modify our techniques to the family of $L$-functions associated to the
quadratic twists of an elliptic curve over $\Q$. To keep things explicit,
we focus on the elliptic curve of conductor 11:
\begin{equation}
    \label{eq:E11}
    E_{11a}: y^2+y =x^3-x.
\end{equation}
The $L$-function of $E_{11a}$ is given by an Euler product of the form
\begin{equation}
    L_{11}(s)=
    \frac{1}{1-11^{-s-1/2 }}
    \prod_{p \neq 11}
    \frac{1}{1-a(p)p^{-s-1/2 }+p^{-2s}},
\end{equation}
which can be expanded into the Dirichlet series
\begin{equation}
    \sum_{n=1}^\infty \frac{a(n)}{n^{1/2 +s}}.
\end{equation}
The Dirichlet series above is absolutely convergent in $\Re{s}>1$.
The coefficients $a(n)$ can be obtained from the Fourier expansion of
the cusp form of weight two and level 11 given by
\begin{equation}
   \sum_{n=1}^\infty a(n) q^n=
    q \prod_{n=1}^\infty (1-q^n)^2 (1-q^{11n})^2,
\end{equation}
or, alternatively, by counting points on $E_{11a}$ over the fields $F_p$, $p$ prime.

The function $L_{11}(s)$ has analytic continuation to all of $\C$ and
satisfies the functional equation
\begin{equation}
   L_{11}(s) = X(s) L_{11}(1-s),
\end{equation}
where
\begin{equation}
   X(s) = \frac{\Gamma(3/2-s)}{\Gamma(s+1/2)}
                       \left(\frac{2\pi}{11^{1/2}}\right)^{2s-1}.
\end{equation}

The $L$-function associated to a quadratic twist of $E_{11a}$ has
a Dirichlet series of the form
\begin{equation}
    L_{11}(s,\chi_d)=\sum_{n=1}^\infty \frac{a(n)}{n^{1/2 +s}} \chi_d(n),
\end{equation}
where $d$ is a fundamental discriminant which we further assume satisfies $(d,11)=1$.
$L_{11}(s,\chi_d)$ satisfies the functional equation
\begin{equation}
    L_{11}(s,\chi_d) = \chi_d(-11)
                       |d|^{1-2s} X(s)
                       L_{11}(1-s,\chi_d).
\end{equation}
When considering the moments of $L_{11}(1/2 ,\chi_d)$ we should restrict
$L(s,\chi_d)$ to have an even functional equation, i.e. $\chi_d(-11)=1$, otherwise
$L(1/2,\chi_d)$ is trivially equal to $0$.
In~\cite{CFKRS}, $d$ was also restricted to being negative
since it allowed them to exploit a theorem of Kohnen and Zagier~\cite{KZ}
to easily gather numerical data for $L_{11}(1/2 ,\chi_d)$ with which to check their
conjecture.

When $d<0$, $\chi_d(-1)=-1$, hence, in order to have an even
functional equation, we require $\chi_d(11)=-1$, i.e.
$d=2,6,7,8,10 \mod 11$. CFKRS conjectured, see Section~5.3 of~\cite{CFKRS}, the
asymptotic expansion:
\begin{equation}
    \sum_{d \in S_-(X) \atop d =  2,6,7,8,10 \mod 11}
    L_{11}(1/2,\chi_d)^k \sim \frac{15}{11\pi^2}
      \frac{1}{X} \int_1^X \Upsilon_k(\log t) dt.
    \label{eqn:L11 CFKRS}
\end{equation}
The extra factor of $5/11$ on the rhs, compared to~\eqref{eq: moment asympt},
reflects the fact that the sum on the left
is over 5 out of 11 possible residue classes mod 11.
Here, $\Upsilon_k$ is the polynomial of degree $k(k-1)/2$ given by
the $k$-fold residue
\begin{equation}
    \label{eq:Upsilon}
    \Upsilon_k(x)=\frac{(-1)^{k(k-1)/2}2^k}{k!} \frac{1}{(2\pi i)^{k}}
    \oint \cdots \oint \frac{R_{11}(z_1,
    \dots,z_{k})\Delta(z_1^2,\dots,z_{k}^2)^2} {\displaystyle
    \prod_{j=1}^{k} z_j^{2k-1}} e^{x \sum_{j=1}^{k}z_j}\,dz_1\dots
    dz_{k} ,
\end{equation}
where
\begin{equation}
    \label{eq:R}
    R_{11}(z_1,\dots,z_k)=A_k(z_1,\dots,z_k)
    \prod_{j=1}^k X(1/2+z_j)^{-1/2}
    \prod_{1\le i < j\le k}\zeta(1+z_i+z_j),
\end{equation}
and, overloading the notation of Section~\ref{sec:CFKRS quadratic},
$A_k$ is the Euler product which is absolutely convergent for
$\sum_{j=1}^k |z_j|<\frac12 $,
\begin{equation}
    A_k(z_1,\dots,z_k) =
    \prod_p R_{11,p}(z_1,\ldots,z_k)
    \prod_{1\le i < j \le k}
    \left(1-\frac{1}{p^{1+z_i+z_j}}\right)
\end{equation}
with, for $p \neq 11$,
\begin{equation}
    R_{11,p} =
           \left(1+\frac 1 p\right)^{-1}
           \left(\frac 1 p +\frac{1}{2}
              \left(
                 \prod_{j=1}^k
                 \frac{1}{1-a(p)p^{-1-z_j}+p^{-1-2z_j}}
               + \prod_{j=1}^k
                 \frac{1}{1+a(p)p^{-1-z_j}+p^{-1-2z_j}}
              \right)
           \right)
\end{equation}
and
\begin{equation}
    R_{11,11} =
    \prod_{j=1}^k  \frac{1}{1+11^{-1-z_j}}.
\end{equation}
Note that, although here we are working with the specific elliptic curve
$E_{11a}$, CFKRS' recipe provides a similar conjecture for the quadratic twists of any elliptic
curve over $\Q$. For many examples, see the paper~\cite{CPRW}. The only difference
is in the conductor, in the local factors of $A_k$ for the primes dividing the
conductor, and in the allowed residue classes (and modulus) for $d$.

Next, $\Upsilon_k(x)$ is a polynomial of degree of $k(k-1)/2$ given by the
$k$-fold residue \reff{eq:Upsilon}. The degree works out smaller compared to
$Q_{\pm}(k,x)$ because the product of zetas in~\eqref{eq:R} involves fewer
zetas, i.e. the product over $i<j$ has $k \choose 2$ factors. Therefore,
we can write
\begin{equation} \label{eq:Upsilonexpansion}
    \Upsilon_k(x) = c_0(k) x^{k(k-1)/2} + c_1(k) x^{k(k-1)/2 - 1} + \ldots + c_{k(k-1)/2}(k).
\end{equation}
Also note that the exponential in~\eqref{eq:Upsilon} has an $x$ rather than $x/2$. This will impact
the powers of 2 that enter into our formulas for the coefficients $c_r(k)$.

To address the poles coming from the zeta-product $\prod \zeta(1+z_i+z_j)$ we absorb some
of the factors of $\Delta(z_1^2, \ldots, z_k^2) = \prod_{1 \leq i < j \leq k} (z_j - z_i)(z_j+z_i)$.
Thus,
\begin{align} \nonumber
    \Upsilon_k (x) & = \frac{(-1)^{k(k-1)/2}2^k}{k!} \frac{1}{(2\pi i)^k} \oint
    \ldots \oint A_k(z_1, \ldots, z_k) \prod_{j=1}^k X(1/2+z_j)^{-1/2}\\ & \times
    \prod_{1 \leq i < j \leq k} \zeta(1 + z_i + z_j) (z_i + z_j)
    \label{eq:NewUpsilonOld} \frac{\Delta(z_1, \ldots, z_k) \Delta(z_1^2, \ldots,
    z_k^2)}{\prod_{j=1}^kz_j^{2k - 1}} \exp(x \sum_{j=1}^k z_j)dz_1 \ldots dz_k. 
\end{align}

We overload notation again and set
\begin{equation}
    a_k:= A_k(0, \ldots, 0)
\end{equation}
and expand
\begin{equation}
    \label{eq:multivariateM} \frac{1}{a_k} A_k(z_1, \ldots, z_k) \prod_{j=1}^k
    X(1/2+z_j)^{-1/2}\prod_{1 \leq i < j \leq k} \zeta(1 + z_i + z_j) (z_i + z_j)
    =: \sum_{j=0}^\infty \sum_{|\lambda| = j} b_\lambda(k) m_\lambda(z),
\end{equation}
where, as before, $m_\lambda(z)$ is the monomial symmetric function for the partition
$\lambda$. The lhs above is holomorphic in a neighbourhood of $z_1=\ldots=z_k=0$,
because the poles from the zeta-product $\prod \zeta(1 + z_i + z_j)$ are
canceled by the product $\prod (z_i + z_j)$. We normalize by $a_k$ so that
the first coefficient is 1.

So \reff{eq:NewUpsilonOld} becomes
\begin{align} \nonumber
    \Upsilon_k (x) &= \frac{(-1)^{k(k-1)/2}2^k}{k!} \frac{a_k}{(2\pi i)^k} \oint
    \ldots \oint \sum_{j=0}^\infty \sum_{|\lambda| = j} b_\lambda(k) m_\lambda(z) \\
    & \times \frac{\Delta(z_1, \ldots, z_k) \Delta(z_1^2, \ldots,
    z_k^2)}{\prod_{j=1}^kz_j^{2k - 1}} \exp(x \sum_{j=1}^k z_j)dz_1 \ldots dz_k.
    \label{eq:NewUpsilon}
\end{align}
Comparing to equation~\eqref{eq:sumofintegrals} we notice three differences:
the extra $2^k$ in front of the integral, the $2k-1$ powers of each $z_j$,
rather than $2k$ powers, in the denominator, and the $x$ rather than $x/2$ in
the exponential. The first two differences are accounted for by the fact that the
product over zetas in~\eqref{eq:Qk} includes $i=j$, and this introduces,
from~\eqref{eq:233}, an extra $2z_j$, for each $j$. Therefore, proceeding as
in Section~\ref{sec:further-lower-order}, we get:
\begin{align} \nonumber
     c_r(k) &= \frac{(-1)^{k(k-1)/2}2^k}{k!} \frac{a_k}{(2\pi i)^k} \oint
     \ldots \oint \sum_{|\lambda| = r} b_\lambda(k) m_\lambda(z) \\
     \label{eq:beforepullingq} & \times \frac{\Delta(z_1, \ldots, z_k)
     \Delta(z_1^2, \ldots, z_k^2)}{\prod_{j=1}^k z_j^{2k - 1}}
     \exp(\sum_{j=1}^k z_j)dz_1 \ldots dz_k.
\end{align}
Analogously to~\eqref{eq:formula for c(r,k)}, we have
\begin{multline}
    \label{eq:formula for c(r,k) E}
    c_r(k) =
        2^k a_k
        \prod_{j=0}^{k-1}\frac{(2j)!}{(k+j-1)!}
    \\
    \times
    \sum_{|\lambda|=r}
    b_\lambda(k)
    \sum_{\bf u}{^{'}} (-1)^{n({\bf u})} E_{\alpha({\bf u})}(k)
    \times (k+i({\bf u})-2)_{u_{i({\bf u})}} (k+i({\bf u})-3)_{u_{i({\bf u})-1}}\dotsm(k-1)_{u_1}.
\end{multline}
where, for a partition $\alpha$,
\begin{equation}
    \label{eq:E_alpha}
     E_\alpha(k)=\det \left( \binom{2k-i-1-\alpha_{k-i+1}}{2k-2j}\right)_{1\leq i,j\leq k}.
\end{equation}
The equation for $c_r(k)$ differs from~\eqref{eq:formula for c(r,k)} in the power of $2$
that appears, and also some of the factorials have an extra $-1$ in them. The latter
comes from the one missing $z_j$ in the denominator of~\eqref{eq:beforepullingq} compared
to~\eqref{eq:Qk}.

Notice that $E_\alpha(k)$ is very similar to $D_\alpha(k)$. The only difference is the
extra $-1$ in the binomial coefficient. We can relate the two determinants by taking advantage
of the entries in the first column of the matrix for $E_\alpha(k)$, which are all 0 except
for the $1,1$ entry. Assume, for now, that $k\geq \max(l(\alpha)+1,\alpha_1)$
(so, in particular, $\alpha_k=0$).
Expanding along the first column, and then reindexing $i,j$ with $i+1,j+1$:
\begin{eqnarray}
    \label{eq:E vs D}
    E_\alpha(k) &=&
    \det \left( \binom{2k-i-1-\alpha_{k-i+1}}{2k-2j}\right)_{2\leq i,j\leq k} \\
    &=&
    \det \left( \binom{2(k-1)-i-\alpha_{(k-1)-i+1}}{2(k-1)-2j}\right)_{1\leq i,j\leq k-1}\\
    &=& D_\alpha(k-1) = 2^{ {k-1 \choose 2} - |\alpha|} \times \ P_\alpha(k-1)
\end{eqnarray}
Also note, while we have assumed that $k>l(\alpha)$, Corollary~\ref{cor:2} tells us
that $P_\alpha(k-1)$, and hence the rhs of~\eqref{eq:E vs D}, vanishes for
$\alpha_1+1 \leq k \leq l(\alpha)$. Furthermore, by the same
method as was used around~\eqref{eq:divisibility part ii},
\begin{equation}
    (k+i({\bf u})-2)_{u_{i({\bf u})}} (k+i({\bf u})-3)_{u_{i({\bf u})-1}}\dotsm(k-1)_{u_1}
\end{equation}
is divisible by $(k-1)\ldots (k-\alpha_1)$, and thus vanishes for $1 \leq k \leq  \alpha_1$.
Therefore, writing
\begin{multline}
    \label{eq:formula for c(r,k) E b}
    c_r(k) =
    2^{k+{k-1\choose 2}-r} a_k \prod_{j=0}^{k-1}\frac{(2j)!}{(k+j-1)!}
    \\
    \times
    \sum_{|\lambda|=r}
    b_\lambda(k)
    \sum_{\bf u}{^{'}} (-1)^{n({\bf u})} P_{\alpha({\bf u})}(k-1)
    \times (k+i({\bf u})-2)_{u_{i({\bf u})}} (k+i({\bf u})-3)_{u_{i({\bf u})-1}}\dotsm(k-1)_{u_1},
\end{multline}
we can replace the requirement that $k\geq \max(l(\alpha)+1,\alpha_1)$ with $k>0$. Finally,
using~\eqref{eq:N}, we get, for $k>0$,
\begin{equation}
    \label{eq:formula for c(r,k) E c}
    c_r(k) =
    2^{k+{k-1\choose 2}-r} a_k
        \prod_{j=0}^{k-1}\frac{(2j)!}{(k+j-1)!}
    \sum_{|\lambda|=r}
    b_\lambda(k) N_\lambda(k-1).
\end{equation}

For example, the $r=0$ term equals
\begin{equation}
    \label{eq:c0 E}
    c_0(k) =
    2^{k+{k-1 \choose 2}} a_k \prod_{j=0}^{k-1}\frac{(2j)!}{(k+j-1)!}.
\end{equation}
One can verify inductively that this matches the leading term as described in (1.5.26)
of~\cite{CFKRS}:
\begin{equation}
    \label{eq:check c0}
    2^{k+{k-1 \choose 2}} \prod_{j=0}^{k-1}\frac{(2j)!}{(k+j-1)!} =
    2^{(k+1)k/2} \prod_{j=0}^{k-1}\frac{j!}{(2j)!}.
\end{equation}

One should also pay attention here that the Taylor coefficients $b_\lambda(k)$, and also $a_k$,
depend on the underlying elliptic curve $E_{11a}$ and its $a(p)$'s. While we
can derive similar formulas for $b_\lambda(k)$ as for quadratic Dirichlet $L$-functions (see the
examples~\eqref{eq:b1},~\eqref{eq:b11}), in order to
accelerate their numerical evaluation we would need to use the symmetric power
$L$-functions associated to the $L$-function $L_{11}(s)$. Acceleration for
$b_\lambda(k)$ will be discussed in a forthcoming paper of Alderson and
Rubinstein~\cite{AR2}.

\begin{table}
  \centering
  \begin{tabular}{|r|c|}
    
\hline
Partition & $P_\lambda(k)$ \\
\hline       	
$ [1] $   &  $ k + 1 $   \\
$ [2] $   &  $ (1/2)  (k + 1)  (k + 2) $   \\
$ [1, 1] $   &  $ (1/2)  (k - 1)  (k + 2) $   \\
$ [3] $   &  $ (1/6)  (k + 1)  (k + 2)  (k + 3) $   \\
$ [2, 1] $   &  $ (1/3)  (k + 2)  (k^2 + k - 3) $   \\
$ [1, 1, 1] $   &  $ (1/6)  (k - 2)  (k - 1)  (k + 3) $   \\
$ [4] $   &  $ (1/24)  (k + 1)  (k + 2)  (k + 3)  (k + 4) $   \\
$ [3, 1] $   &  $ (1/8)  (k + 2)  (k + 3)  (k^2 + k - 4) $   \\
$ [2, 2] $   &  $ (1/12)  (k - 2)  (k + 1)  (k + 2)  (k + 3) $  \\
$ [2, 1, 1] $   &  $ (1/8)  (k - 2)  (k + 3)  (k^2 + k - 4) $  \\
$ [1, 1, 1, 1] $   &  $ (1/24)  (k - 3)  (k - 2)  (k - 1)  (k + 4) $   \\
$ [5] $   &  $ (1/120)  (k + 1)  (k + 2)  (k + 3)  (k + 4)  (k + 5) $   \\
$ [4, 1] $   &  $ (1/30)  (k + 2)  (k + 3)  (k + 4)  (k^2 + k - 5) $   \\
$ [3, 2] $   &  $ (1/24)  (k + 1)  (k + 2)  (k + 3)  (k^2 + k - 8) $   \\
$ [3, 1, 1] $   &  $ (1/20)  (k + 3)  (k^4 + 2k^3 - 11k^2 - 12k + 40) $   \\
$ [2, 2, 1] $   &  $ (1/24)  (k - 2)  (k + 1)  (k + 3)  (k^2 + k - 8) $   \\
$ [2, 1, 1, 1] $   &  $ (1/30)  (k - 3)  (k - 2)  (k + 4)  (k^2 + k - 5) $   \\
$ [1, 1, 1, 1, 1] $   &  $ (1/120)  (k - 4)  (k - 3)  (k - 2) (k - 1)  (k + 5) $   \\
$ [6] $   &  $ (1/720)  (k + 1)  (k + 2)  (k + 3)  (k + 4)  (k + 5)  (k + 6) $   \\
$ [5, 1] $   &  $ (1/144)  (k - 2)  (k + 2)  (k + 4)  (k + 5) (k + 3)^2 $   \\
$ [4, 2] $   &  $ (1/80)  (k + 1)  (k + 2)  (k + 3)  (k + 4) (k^2 + k - 10) $   \\
$ [4, 1, 1] $   &  $ (1/72)  (k + 3)  (k + 4)  (k^4 + 2k^3 - 13k^2 - 14k + 60) $   \\
$ [3, 3] $   &  $ (1/144)  (k - 3)  (k + 1)  (k + 3)  (k + 4) (k + 2)^2 $   \\
$ [3, 2, 1] $   &  $ (1/45)  (k + 1)  (k + 3)  (k^4 + 2k^3 - 16k^2 - 17k + 75) $   \\
$ [3, 1, 1, 1] $   &  $ (1/72)  (k - 3)  (k + 4)  (k^4 + 2k^3 - 13k^2 - 14k + 60) $   \\
$ [2, 2, 2] $   &  $ (1/144)  (k - 3)  (k - 2)  (k - 1)  (k + 2) (k + 3)  (k + 4) $   \\
$ [2, 2, 1, 1] $   &  $ (1/80)  (k - 3)  (k - 2)  (k + 1)  (k + 4)  (k^2 + k - 10) $   \\
$ [2, 1, 1, 1, 1] $   &  $ (1/144)  (k - 4)  (k - 3)  (k + 3) (k + 5)  (k - 2)^2 $   \\
$ [1, 1, 1, 1, 1, 1] $   &  $ (1/720)  (k - 5)  (k - 4)  (k - 3) (k - 2)  (k - 1)  (k + 6) $   \\
$ [7] $ & $ (1/5040)  (k + 1)  (k + 2)  (k + 3)  (k + 4)  (k + 5)  (k + 6)  (k + 7) $ \\
$ [6, 1] $ & $ (1/840)  (k + 2)  (k + 3)  (k + 4)  (k + 5)  (k + 6)  (k^2 + k - 7) $ \\
$ [5, 2] $ & $ (1/360)  (k - 3)  (k + 1)  (k + 2)  (k + 3)  (k + 5)  (k + 4)^2 $ \\
$ [5, 1, 1] $ & $ (1/336)  (k + 3)  (k + 4)  (k + 5)  (k^4 + 2k^3 - 15k^2 - 16k + 84) $ \\
$ [4, 3] $ & $ (1/360)  (k + 1)  (k + 3)  (k + 4)  (k + 2)^2 (k^2 + k - 15) $ \\
$ [4, 2, 1] $ & $ (1/144)  (k + 1)  (k + 3)  (k + 4)  (k^4 + 2k^3 - 19k^2 - 20k + 108) $ \\
$ [4, 1, 1, 1] $ & $ (1/252)  (k + 4)  (k^6 + 3k^5 - 26k^4 - 57k^3 + 277k^2 + 306k - 1260) $ \\
$ [3, 3, 1] $ & $ (1/240)  (k - 3)  (k + 1)  (k + 2)  (k + 3) (k + 4)  (k^2 + k - 10) $ \\
$ [3, 2, 2] $ & $ (1/240)  (k - 3)  (k - 1)  (k + 2)  (k + 3) (k + 4)  (k^2 + k - 10) $ \\
$ [3, 2, 1, 1] $ & $ (1/144)  (k - 3)  (k + 1)  (k + 4)  (k^4 + 2k^3 - 19k^2 - 20k + 108) $ \\
$ [3, 1, 1, 1, 1] $ & $ (1/336)  (k - 4)  (k - 3)  (k + 5)  (k^4 + 2k^3 - 15k^2 - 16k + 84) $ \\
$ [2, 2, 2, 1] $ & $ (1/360)  (k - 3)  (k - 2)  (k - 1)  (k + 2) (k + 4)  (k^2 + k - 15) $ \\
$ [2, 2, 1, 1, 1] $ & $ (1/360)  (k - 4)  (k - 2)  (k + 1)  (k + 4)  (k + 5)  (k - 3)^2 $ \\
$ [2, 1, 1, 1, 1, 1] $ & $ (1/840)  (k - 5)  (k - 4)  (k - 3) (k - 2)  (k + 6)  (k^2 + k - 7) $ \\
$ [1, 1, 1, 1, 1, 1, 1] $ & $ (1/5040)  (k - 6)  (k - 5)  (k - 4) (k - 3)  (k - 2)  (k - 1)  (k + 7) $ \\
\hline
  \end{tabular}
  \caption{Table of $P_\lambda(k)$}
  \label{tab:plambda}
\end{table}

\begin{table}

\begin{tabular}[t]{cc c c}
\toprule[2pt]
$r$ & $ c_-(r,1)$ & $ c_-(r,2)$& $ c_-(r,3)$ \\
\midrule
\begin{tabular}[t]{r}
  0\\ 1\\ 2\\ 3\\ 4\\ 5\\ 6
\end{tabular} &
  \begin{tabular}[t]{c}
3.522211004995827732e-01
 \\
6.175500336140218316e-01
  \end{tabular} &
  \begin{tabular}[t]{c}
  1.238375103096108452e-02
 \\
1.807468351186638511e-01
 \\
3.658991414081511628e-01
 \\
-1.398953902867718369e-01
  \end{tabular} &
  \begin{tabular}[t]{c}
1.528376099282021425e-05
 \\
8.968276397996084726e-04
 \\
1.701420175947633562e-02
 \\
1.093281830681910732e-01
 \\
1.358556940901993748e-01
 \\
-2.329509111366616925e-01
 \\
4.735303837788046866e-01
  \end{tabular}\\
\bottomrule[2pt]\\
$r$& $ c_-(r,4)$ & $ c_-(r,5)$& $ c_-(r,6)$ \\
\midrule \\
\begin{tabular}[t]{r}
  0\\ 1\\ 2\\ 3\\ 4\\ 5\\ 6\\7 \\ 8 \\ 9 \\ 10
\end{tabular} &
\begin{tabular}[t]{c}
3.158268332443340154e-10
 \\
5.062201340608140133e-08
 \\
3.252070477914552180e-06
 \\
1.065078255299183117e-04
 \\
1.865791348720969960e-03
 \\
1.658674128885722146e-02
 \\
5.985999910494527870e-02
 \\
5.231179842747744717e-03
 \\
-1.097356193524353096e-01
 \\
5.581253300381869842e-01
 \\
1.918594095122517496e-01
\end{tabular}
&
\begin{tabular}[t]{c}
6.712517611066278238e-17
 \\
2.341233253582258184e-14
 \\
3.571169234103129887e-12
 \\
3.127118490785452708e-10
 \\
1.734617312939144360e-08
 \\
6.342941105701246722e-07
 \\
1.541064437383931078e-05
 \\
2.441498848686470880e-04
 \\
2.390928284573956911e-03
 \\
1.275610736275904766e-02
 \\
2.430382016767882944e-02
\end{tabular} &
\begin{tabular}[t]{c}
1.036004645427003276e-25
 \\
6.796814066740219201e-23
 \\
2.037808336505920108e-20
 \\
3.698051408075659748e-18
 \\
4.534838798273249707e-16
 \\
3.972866885083416336e-14
 \\
2.563279107875100164e-12
 \\
1.237229229636910631e-10
 \\
4.491515829566301398e-09
 \\
1.222154548508955419e-07
 \\
2.461203700713661380e-06
\end{tabular}\\
\bottomrule[2pt]
$r$& $ c_-(r,7)$ & $c_-(r,8)$& $ c_-(r,9)$ \\
\midrule\\
\begin{tabular}[t]{r}
  0\\1\\2\\3\\4\\5\\6\\7\\8\\9\\10
\end{tabular}
&
\begin{tabular}[t]{c}
8.864927187204894781e-37
 \\
9.894437508330137269e-34
 \\
5.176293026015439716e-31
 \\
1.686724585610585967e-28
 \\
3.837267516078630273e-26
 \\
6.474635477336820480e-24
 \\
8.402114103039537077e-22
 \\
8.581764459399681586e-20
 \\
7.002464589632248733e-18
 \\
4.607034349981096374e-16
 \\
2.455973970379903840e-14
\end{tabular} &
\begin{tabular}[t]{c}
3.372009502181036150e-50
 \\
5.951191608649093822e-47
 \\
5.002043249634522587e-44
 \\
2.664702289380503418e-41
 \\
1.010164553397544484e-38
 \\
2.900498887294046119e-36
 \\
6.555588245821587108e-34
 \\
1.196609980002393296e-31
 \\
1.795828629692653400e-29
 \\
2.244368542496810519e-27
 \\
2.357312576663548340e-25
\end{tabular} &
\begin{tabular}[t]{c}       	
4.727735796587526113e-66
 \\
1.248019487993274422e-62
 \\
1.585820955757896443e-59
 \\
1.291823649274241834e-56
 \\
7.580660624239738211e-54
 \\
3.413900516458523702e-51
 \\
1.227404779731471396e-48
 \\
3.618608212113140382e-46
 \\
8.916974338520402569e-44
 \\
1.862786263819570034e-41
 \\
3.334524507937658586e-39
\end{tabular}\\
\bottomrule[2pt]
\end{tabular}

\caption[$c_-(r,k)$]
{The coefficients $ c_-(r,k)$ of $Q_-(k)$.
}\label{tab:cminusrk}
\end{table}

\begin{table}
  \centering

  \begin{tabular}{cccc}

\toprule[2pt]\\
$r$ & $ c_+(r,1)$ & $ c_+(r,2)$& $ c_+(r,3)$ \\
\midrule\\
\begin{tabular}[t]{r}
0\\  1\\ 2\\ 3\\ 4\\ 5\\ 6\
\end{tabular}
&
    \begin{tabular}[t]{c}	
3.522211004995827732e-01
 \\
-4.889851881547797041e-01
    \end{tabular} &
    \begin{tabular}[t]{c}
  1.238375103096108452e-02
 \\
6.403273133040673915e-02
 \\
-4.030985462971436450e-01
 \\
8.784723252866324383e-01       	
   \end{tabular} &
    \begin{tabular}[t]{c}
 1.528376099282021425e-05
 \\
6.087355322740111135e-04
 \\
5.189536257221761054e-03
 \\
-2.070416696161206729e-02
 \\
-4.836560144295628388e-02
 \\
6.305676273169569246e-01
 \\
-1.231149543676485214
  \end{tabular}\\

\bottomrule[2pt]\\
$r$ & $c_+(r,4)$ & $ c_+(r,5)$& $c_+(r,6)$ \\
\toprule\\

\begin{tabular}[t]{r}
0\\  1\\ 2\\ 3\\ 4\\ 5\\ 6\\ 7 \\ 8 \\ 9 \\ 10
\end{tabular}
&

\begin{tabular}[t]{c}
 3.158268332443340154e-10
 \\
4.070002081481211197e-08
 \\
1.961035634727995841e-06
 \\
4.187933734218812260e-05
 \\
3.233832982317403053e-04
 \\
-7.264209058002128044e-04
 \\
-9.741303115420443803e-03
 \\
6.254058547607513341e-02
 \\
5.338039400180279170e-02
 \\
-1.125787514381924481e+00
 \\
2.125417457224375362
\end{tabular} &

\begin{tabular}[t]{c}
6.712517611066278238e-17
 \\
2.024913313371989448e-14
 \\
2.611003455556346309e-12
 \\
1.870888923760240058e-10
 \\
8.086250862410257040e-09
 \\
2.126496335543600159e-07
 \\
3.194157049041922835e-06
 \\
2.120198748289444789e-05
 \\
-3.390055513847315853e-05
 \\
-7.750613901748660065e-04
 \\
3.339978554290242568e-03
\end{tabular} &
\begin{tabular}[t]{c}
1.036004645427003276e-25
 \\
6.113326104276961713e-23
 \\
1.632224321325099403e-20
 \\
2.605311255686981285e-18
 \\
2.766415183453526818e-16
 \\
2.056437432501927988e-14
 \\
1.095709499896029594e-12
 \\
4.206172871179562219e-11
 \\
1.149109718292255815e-09
 \\
2.154509460431619112e-08
 \\
2.543371224701971233e-07
\end{tabular} \\

\bottomrule[2pt]\\
$r$ & $ c_+(r,7)$ & $ c_+(r,8)$& $ c_+(r,9)$ \\
\midrule\\
\begin{tabular}[t]{r}
0  \\1\\ 2\\ 3\\ 4\\ 5\\ 6\\ 7 \\ 8 \\ 9 \\ 10
\end{tabular}
&
\begin{tabular}[t]{c}
 8.864927187204894781e-37
 \\
9.114637784804059894e-34
 \\
4.370089613567423486e-31
 \\
1.297363094463138851e-28
 \\
2.670392092372496088e-26
 \\
4.043466811338890795e-24
 \\
4.663148139710778893e-22
 \\
4.183154331210266578e-20
 \\
2.954857264190019988e-18
 \\
1.652770327042906306e-16
 \\
7.319238365079051443e-15
\end{tabular} &
\begin{tabular}[t]{c}
3.372009502181036150e-50
 \\
5.569826318573164385e-47
 \\
4.368642207198861832e-44
 \\
2.164658555649376388e-41
 \\
7.604817314362535383e-39
 \\
2.015327809331532264e-36
 \\
4.184593239584908611e-34
 \\
6.980465161514108456e-32
 \\
9.516651650236242059e-30
 \\
1.073015400698217206e-27
 \\
1.008662233782716849e-25
\end{tabular} &
\begin{tabular}[t]{c}
 4.727735796587526113e-66
 \\
1.181182697783246367e-62
 \\
1.417926553457661234e-59
 \\
1.089051480593133551e-56
 \\
6.012641112088390226e-54
 \\
2.541594397695401893e-51
 \\
8.555207141044511720e-49
 \\
2.354807833463352272e-46
 \\
5.400892227120418237e-44
 \\
1.046573394851932219e-41
 \\
1.731269798305270612e-39
\end{tabular}\\
\bottomrule[2pt]

  \end{tabular}
  \caption{The coefficients $c_+(r,k)$ of $Q_+(r,k)$.}
  \label{tab:cplusrk}
\end{table}

\vspace*{20mm}
\bibliography{mybibliography}{}
\bibliographystyle{amsalpha}
\vspace*{20mm}

\end{document}